\documentclass[11pt]{article}
\usepackage{graphicx} % Required for inserting images
\usepackage[margin=1in]{geometry}
\usepackage[utf8]{inputenc}
\usepackage[english]{babel}
\usepackage{amssymb}
\usepackage{amsmath}
\usepackage{amsfonts}
\usepackage{bbm}
\usepackage{booktabs}
\usepackage{array}
\usepackage{verbatim}
\usepackage{xcolor}
\usepackage{geometry}
\usepackage{setspace}
\usepackage{amsthm}
\graphicspath{ {./images/} }
\usepackage{graphicx}
\usepackage{float} 
\usepackage{subfigure}
\usepackage{enumitem}
\usepackage{float}
\usepackage{csquotes}

\usepackage{setspace}
\usepackage{authblk}
\usepackage{orcidlink}

\DeclareMathOperator*{\argmin}{arg\,min}

\usepackage[linesnumbered,ruled,vlined]{algorithm2e} 

\theoremstyle{definition}
\newtheorem{theorem}{Theorem}[section]

\newtheorem{assumption}[theorem]{Assumption} 
\newtheorem{lemma}[theorem]{Lemma}
\newtheorem{remark}[theorem]{Remark}
\newtheorem{proposition}[theorem]{Proposition}

\numberwithin{equation}{section}%公式按章节编号

\title{Convergence of Sinkhorn's Algorithm for Entropic Martingale Optimal Transport Problem}

\author[12]{Fan CHEN\,\orcidlink{0000-0003-0082-7908}}
\author[3]{Giovanni CONFORTI\,\orcidlink{0000-0001-5939-1016}}
\author[4]{Zhenjie REN\,\orcidlink{0000-0003-4656-4074}}
\author[2]{Xiaozhen WANG\,\orcidlink{0009-0009-9286-0094}}
\affil[1]{School of Mathematical Sciences, Shanghai Jiao Tong University, Shanghai, China}
\affil[2]{CEREMADE, Université Paris-Dauphine, PSL, Paris, France}
\affil[3]{Dipartimento Matematica, Università degli Studi di Padova, Padova, Italy}
\affil[4]{LaMME, Université Évry Paris-Saclay, Évry-Courcouronnes, France}

\begin{document}

\maketitle
\begin{abstract}
    In this paper, we study the Entropic Martingale Optimal Transport (EMOT) problem on $\mathbb{R}$. The investigation of the EMOT problem arises in the calibration problem of the Stochastic Volatility Models, where martingale constraints reflect no-arbitrage pricing conditions under the risk-neutral measure, as originally proposed by Henry-Labordère. We first establish the dual formulation of the EMOT problem and prove that Sinkhorn’s algorithm achieves an exponential convergence rate under mild conditions. Notably, our analysis does not presuppose the existence of optimal potentials and rigorously confirms the absence of a primal-dual gap. These results provide a theoretical foundation for solving EMOT via Sinkhorn’s method and constructing the optimal distribution from dual coefficients.
\end{abstract}

% Add a footnote to the first page
\renewcommand{\thefootnote}{\fnsymbol{footnote}}
\footnotetext[1]{The third named author’s research was supported by the Finance For Energy Market Research Centre,  the France 2030 grant (ANR-21-EXES-0003), and PEPR PDE-AI project.}

\section{Introduction}
The Optimal Transport (OT) problem has been extensively studied in the literature \cite{santambrogio2015optimal,villani2009optimal}, addressing fundamental aspects such as well-posedness, representation, and geometric perspectives, and its applications can be found in widespread fields such as computer vision and data science \cite{arjovsky2017wasserstein, flamary2021pot, peyre2019computational}. The core objective of OT is to determine the most cost-effective method for transporting one distribution to another. The Martingale Optimal Transport (MOT) problem, emerging in the last decade, follows from the first introduction by Beiglböck et al. \cite{beiglbock2013model} and Galichon et al. \cite{galichon2014stochastic}, further requires that the transport itself satisfies the martingale property: let $\mu$, $\nu$ be two probability measures and $c(\cdot)$ be a point-to-point transport cost, we aim to minimize over $\pi\in\mathcal{M}_0(\mu,\nu)\subset \mathcal P(\mathcal X\times\mathcal Y)$,
\begin{equation}
\label{eq:primal_intro_without_reg}
    \inf_{\pi\in \mathcal{M}_0(\mu,\nu)} \int_{\mathcal X\times \mathcal Y} c(x,y)d\pi,
\end{equation}
over all joint distributions with marginal distribution $\mu$ and $\nu$ and the martingale property, i.e., 
\begin{equation*}
    \int_{A\times \mathcal Y}d\pi = \mu(A), \quad \int_{\mathcal X\times B}d\pi = \nu(B), \quad \int_{\mathcal Y}(y-x)d\pi_{y|x} = 0,
\end{equation*}
for any Borel sets $A\in \mathcal B(\mathcal X),B\in \mathcal B(\mathcal Y)$, where $\pi_{y|x}$ is the conditional probability measure defined via the disintegration $\pi = \mu \cdot \pi_{y|x}$; that is, for $\mu-$almost every $x \in \mathcal{X}$, the martingale condition holds.
According to Strassen's theorem \cite{strassen1965existence}, for the 1-D case when $\mathcal X = \mathcal Y =\mathbb R$, the existence of a martingale transport between $\mu$ and $\nu$ (assume the first order moment is finite) is equivalent to $\mu \preceq_c \nu$, which means that $\mu$ and $\nu$ satisfy the convex order, that is $\int_{\mathbb{R}} f(x) \mu(d x) \leq \int_{\mathbb{R}} f(y) v(d y)$
for every convex function $f: \mathbb{R} \rightarrow (-\infty, \infty]$. 

As a brief review of the literature, we mention that the martingale optimal transport problem is recognized as the dual of the model-free super-hedging problem \cite{de2018entropic, dolinsky2014martingale}. Further investigations have focused on duality theory, elaborated in works by Bartl et al., Cheridito et al., Guo et al., and Hou et al. \cite{bartl2019duality,cheridito2017duality,guo2016monotonicity, hou2018robust}, with a comprehensive list of references provided by Cheridito et al. \cite{cheridito2021martingale}. Additionally, the optimal Skorokhod embedding problem aligns with the continuous-time MOT through time changes, as explored by Beiglböck et al. \cite{beiglbock2017optimal,beiglbock2021fine}. Research on the stability of MOT, conducted by Backhoff et al. and Wiesel \cite{backhoff2022stability, wiesel2019continuity}, offers theoretical insights into the errors that may arise from statistical distribution estimates. In \cite{doldi2023entropy,doldi2024entropy}, the authors study the MOT with entropic penalty.
On a computational front, Guo et al. proposed methodologies for solving the MOT problem using discretization, relaxation of the marginal condition, and conversion to linear programming \cite{guo2019computational}. More recently, Hiew, Nenna, and Pass~\cite{hiew2024ordinary} proposed an ODE-based method for solving entropy-regularized MOT problems. Although not using Sinkhorn iterations, their approach highlights the flexibility of entropic methods in handling martingale constraints. We shall mention more connections between MOT and its applications in finance in Section \ref{sec:applications}. 

Slightly different from \eqref{eq:primal_intro_without_reg}, in this paper the state spaces are set up as $\mathcal X, \mathcal Y, \mathcal Z = \mathbb R$. The inclusion of the $\mathcal{Z}$ space allows us to incorporate an additional random variable, for example random volatility, into the model, while the asset price variables at two future time points are modeled within the $\mathcal{X} \times \mathcal{Y}$ space. We define $\Pi(\mu, \nu)$ as the set of all joint distributions on the space $\mathcal{X} \times \mathcal{Y} \times \mathcal{Z}$, with the condition that their marginal distributions on $\mathcal{X}$ is $\mu$ and on $\mathcal{Y}$ is $\nu$:
\begin{equation*}
\begin{aligned}
\Pi(\mu, \nu)=\{\pi \in \mathcal{P}(\mathcal X\times \mathcal Y\times\mathcal Z),\ s.t.\ \pi(A \times \mathcal Y\times\mathcal Z)=\mu(A), \pi(\mathcal X \times &B\times \mathcal Z)=\nu(B),\\
&\ \text { for any } A,B \in \mathcal B(\mathbb R)\}.
\end{aligned}
\end{equation*}
We denote the set of all joint distributions in $\Pi(\mu,\nu)$ which further exhibits the martingale property between spaces $\mathcal X$ and $\mathcal Y$ to be
\begin{equation}\label{eq:def-M}
    \mathcal M(\mu,\nu) := \Big\{\pi \in  \Pi(\mu,\nu),\ s.t.\ \int_{\mathcal Y\times\mathcal Z} (y-x)\pi_{y,z|x}(dy) = 0, \ \ \mu{\text{-a.s.}}\ x \in \mathcal X\Big\}.
\end{equation}
Given $\rho \in \mathcal P(\mathcal Z)$ and a prior reference Gibbs probability measure $Q\in\mathcal P(\mathcal X\times\mathcal Y\times \mathcal Z)$ such that $dQ/d(\mu\otimes \nu\otimes\rho) = \exp(-c(x,y,z))$, where we implicitly assume $\int \exp(-c(x,y,z))  d\mu\otimes \nu \otimes \rho = 1$, our EMOT problem is to solve: 
\begin{equation}
\label{eq:primal_intro}
    \inf_{\pi\in \mathcal M(\mu,\nu)} H(\pi|Q) = \inf_{\pi \in \mathcal M(\mu,\nu)} H(\pi|\mu \otimes \nu\otimes \rho) + \int_{\mathcal X\times \mathcal Y\times \mathcal Z} c(x,y,z)d\pi,
\end{equation}
where $H(\pi|Q)$ represents the relative entropy with respect to reference measure $Q$. Compared to \eqref{eq:primal_intro_without_reg}, in our EMOT problem formulation \eqref{eq:primal_intro}, we replace the linear cost function on $\pi$ by a nonlinear one. 

\begin{remark}
\label{rem:rho}
The auxiliary space $\mathcal Z$ provides additional flexibility in the modeling. It allows for incorporating latent factors such as stochastic volatility, regime changes, and other time-dependent features. For example, in the context of Stochastic Volatility models, the spaces $\mathcal X$ and $\mathcal Y$ accommodate the asset price variables at time $T_1$ and $T_2$ (with $T_1<T_2$), while space $\mathcal Z$ corresponds to the volatility variable at $T_1$, see Section \ref{sec:calibration}. Indeed, the probability measure $\rho$ on the $\mathcal Z$ space is not a marginal constraint of the joint distribution; rather, it serves as a base probability measure that jointly supports the reference measure $Q$. 

From a technical perspective, introducing \(\rho\) allows the reference measure \(Q\) to factorize with respect to \(\mu \otimes \nu \otimes \rho\), which simplifies the entropy term and facilitates the implementation of Sinkhorn updates.
\end{remark}

\begin{remark}
We can decompose the relative entropy into two parts:
\begin{equation*}
    H(\pi|Q) = \int_{\mathcal X\times \mathcal Y} \log \frac{d\pi_{x,y}}{dQ_{x,y}}(x,y)d\pi_{x,y} + \int_{\mathcal X\times \mathcal Y}\Big(\int_{\mathcal Z} \log \frac{d\pi_{z|x,y}}{dQ_{z|x,y}}(z)d\pi_{z|x,y}\Big)d\pi_{x,y},
\end{equation*}
and we find that the optimality is attained when $\pi_{z|x,y} = Q_{z|x,y}$, $\pi_{x,y}-$a.s.. 
\end{remark}

When the martingale constraint is removed, this optimization problem is referred to as the Entropic Optimal Transport (EOT) problem. Initially considered an approximation to the traditional optimal transport problem, EOT is detailed further, including historical context and references, in \cite{peyre2019computational}. The EOT problem can be efficiently solved using Sinkhorn's algorithm, introduced in the seminal work by Cuturi~\cite{cuturi2013sinkhorn}; around the same time, Galichon and Salanié~\cite{galichon2022cupid} independently explored entropy-regularized formulations in matching models, marking an early development of entropic methods in optimal transport.
Subsequent developments include iterative Bregman projections \cite{benamou2015iterative}, scaling methods for unbalanced OT \cite{chizat2018scaling}, and stabilized sparse algorithms \cite{schmitzer2019stabilized}, which significantly broadened the applicability and stability of Sinkhorn-type solvers, making them especially effective in approximating Monge–Kantorovich optimal transport solutions, particularly in high-dimensional machine learning tasks where the regularization intensity is gradually reduced~\cite{peyre2019quantum, solomon2015convolutional}.

The primal problem of EMOT \eqref{eq:primal_intro} can be connected with its dual problem, which reads
\begin{equation}
\label{eq:dual_intro}
   \sup_{f\in L^\infty(\mu),g\in L^\infty(\nu),h\in L^\infty(\mu)} \mathcal G(f,g,h) := - Z(f,g,h) - \int_{\mathcal X} f(x)\mu(dx) - \int_{\mathcal Y} g(y) \nu(dy),
\end{equation}
where  $Z(f,g,h)$ is the total mass of a positive measure which will be precised later.
% \begin{equation*}
%         \frac{d\pi(f,g,h)}{d(\mu\otimes \nu\otimes \rho)} = \exp\big(-c(x,y,z)-f(x)-g(y) - h(x) (y-x)\big), \quad Z(f,g,h):= \pi(f,g,h)(\mathcal X\times\mathcal Y\times\mathcal Z).
% \end{equation*}
We are going to solve the dual problem \eqref{eq:dual_intro}  via Sinkhorn's algorithm, named after Sinkhorn and Knopp by their work \cite{sinkhorn1967concerning},  a coordinate ascent algorithm that can be written as
\begin{equation}
\label{eq:sinkhorn_alg_min_conti}
\begin{aligned}
    &\tilde{h}^{n+1} = \argmin_{h\in L^{\infty}(\mu)} \mathcal G(f^n,g^n,h), \\
    &\tilde{f}^{n+1} = \argmin_{f\in L^{\infty}(\mu) } \mathcal G(f,g^n,{\tilde h}^{n+1}), \\
    &\tilde{g}^{n+1} = \argmin_{g\in L^{\infty}(\nu)} \mathcal G(\tilde{f}^{n+1},g,\tilde{h}^{n+1}), \\
\end{aligned}
\end{equation}
and we normalize $(\tilde{f}^{n+1},\tilde{g}^{n+1},\tilde{h}^{n+1}) $ as
\begin{equation*}
    (f^{n+1},g^{n+1},h^{n+1}) = \big(\tilde{f}^{n+1}+ \lambda_{\tilde{g}^{n+1}}-\lambda_{\tilde{h}^{n+1}} x,\tilde{g}^{n+1}-\lambda_{\tilde{g}^{n+1}}+\lambda_{\tilde{h}^{n+1}} y,\tilde h^{n+1}-\lambda_{\tilde{h}^{n+1}}\big),
\end{equation*}
where the normalization coefficients $\lambda_{\tilde{g}^{n+1}} = \int_{\mathcal Y}\tilde{g}^{n+1}(y)\nu(dy)$, $\lambda_{\tilde{h}^{n+1}} = \int_{\mathcal X}\tilde{h}^{n+1}(x)\mu(dx)$, and the initial functions $(f^0,g^0,h^0)$ are given and satisfy some mild conditions, see Algorithm \ref{alg:sinkhorn_alg} for the pseudo-code. Sinkhorn's algorithm is renowned for its capability to address both two-marginal and multi-marginal EOT problems, achieving exponential convergence rates as demonstrated in \cite{carlier2022linear, cuturi2013sinkhorn, ghosal2022convergence}. Further research can be found in relaxing the regularity assumption on the cost $c(\cdot)$ \cite{chen2016entropic,nutz2021introduction,nutz2023stability,ruschendorf1995convergence}. However, the martingale constraint introduces additional difficulties to the convergence analysis, primarily due to the dual coefficient $h$, as discussed in \cite{nutz2024martingale}. Martingale Sinkhorn's algorithm, similar to \eqref{eq:sinkhorn_alg_min_conti} without the $\mathcal{Z}$-component, was first formulated in \cite{de2018entropic,de2019building}. While many numerical experiments exhibit the appealing performance of the martingale Sinkhorn's algorithm in fast convergence, for example, \cite{de2018entropic, de2019building, guyon2024dispersion, guyon2022fast},  a rigorous theoretical guarantee was still at large. 
Our convergence analysis follows the general framework of coordinate ascent methods as in~\cite{beck2013convergence}, and also relates to recent results on multimarginal Sinkhorn convergence~\cite{carlier2022linear}, while addressing the additional difficulty from the martingale constraint.

In this paper we make an effort to fill the gaps in the convergence analysis in \cite{de2018entropic, de2019building} and in particular avoid assuming a priori the absence of a primal–dual gap.  Our main Theorem \ref{thm:existence_minimizer_linear_convergence} is four-folded: 
\begin{itemize}
    \item We prove the exponential convergence of Sinkhorn's algorithm solving the dual problem \eqref{eq:dual_intro};
    \item  We show the existence of the optimal dual potentials $f^*,h^*\in L^\infty(\mu)$, $g^*\in L^\infty(\nu)$ that attain the maximum of the dual problem \eqref{eq:dual_intro}, and the convergence is in the $L^\infty$ norm, i.e., $f^n\to f^*$, $h^n\to h^*$ in $L^\infty(\mu)$ and $g^n\to g^*$ in $L^\infty(\nu)$;
    \item The dual functions $f^n,g^n,h^n$
    converges exponentially to $f^*,g^*,h^*$ in the $L^2$ norm;
    \item  We verify that the induced probability measure $\pi(f^*,g^*,h^*)$ is a solution to the primal problem \eqref{eq:primal_intro} and ensure the absence of the primal-dual gap between \eqref{eq:primal_intro} and \eqref{eq:dual_intro}.
\end{itemize} 

Besides \cite{de2018entropic, de2019building}, the most relevant research to us is the recent work \cite{nutz2024martingale}, in which the authors proved the existence of the optimal dual potentials as well as the absence of the primal-dual gap using a compactness argument. Notably, in their argument the authors introduced an auxiliary problem equivalent to the dual one, which helps to prove uniform local boundedness of $(h^n)_n$ $\mu$-a.s.. 
However, their work is currently limited to the case that the transport cost $c(\cdot)$ is a constant. 

Our paper is structured as follows: In Section \ref{sec:main_result}, we rigorously define the Sinkhorn iterations, and then state our main convergence result including the uniform boundedness of potentials $(f^n,g^n,h^n)_n$ and their subsequent exponential convergence. A numerical experiment in the context of quantitative finance is included in Section \ref{sec:applications}. Detailed proofs of our main results are laid out in Section \ref{sec:main_result_proofs}.

\subsection{Notations}
We use $\mathcal B(\mathbb R)$ to represent all Borel sets in $\mathbb R$, and use $\mathcal P(\mathbb R)$ to denote all the probability measures on $\mathbb R$ with $\sigma$-field $\mathcal B(\mathbb R)$, and $\mathcal M^+(\mathbb R)$ for the set of all positive measures in $\mathbb R^d$. For two measures $\mu$ and $\nu$, we use the notation $\mu\ll\nu$ if $\mu$ is absolutely continuous with respect to $\nu$, and $\mu\sim\nu$ means $\mu$ and $\nu$ are equivalent. We denote $\mu \otimes \nu$ as the product measure of $\mu$ and $\nu$. The support of a measure $\mu$ is denoted as $\textrm{supp}(\mu)$. We use $C^0$, $C^1$ to denote the function space for all functions which are continuous or continuously differentiable. 

A vector is denoted as $x = (x_1,x_2,...,x_d)\in\mathbb R^d$. We use $|x|$ to denote the Euclidean distance. For a measure $\mu\in \mathcal P (\mathbb R^d)$, we denote its marginal distribution on its $i$-th dimension as $\mu_i$. We consider the transport between state spaces $\mathcal X$ and $\mathcal Y$, and for $x\in \mathcal X$, $y\in\mathcal Y$, we use $\pi_{y|x}$ to denote the conditional distribution of $y$ given $x$. We denote the function space $L^p(\mu)$ to be the $L^p$ space with respect to measure $\mu$, and the $L^\infty(\mu)$ to represent the space of all functions essentially bounded under measure $\mu$. Sometimes, we may omit the measure in the bracket when there is no ambiguity.

The relative entropy is defined as 
\begin{equation*}
    H(\pi|Q) = \begin{cases}
        \int \log \frac{d\pi}{dQ} d\pi, & \pi \ll Q,\\
        +\infty, & \mbox{elsewhere}.
    \end{cases}
\end{equation*}
In particular, $H(\pi|Q) \ge 0$ and $H(\pi|Q) = 0$ if and only if $\pi = Q$.

We denote by $[d] := \{1,2,...,d\}$, and by $\mathbb N$ (resp. $\mathbb N^+$)  the set of (resp. positive) natural numbers.
\section{Main Results}
\label{sec:main_result}
Recall our state space $\mathcal X\times \mathcal Y\times \mathcal Z=\mathbb R^3$ and probability measures $\mu,\nu,\rho\in\mathcal P(\mathbb R)$, the EMOT problem can be formulated as
\begin{equation}
    \label{eq:mg_transport}
    \inf_{\pi \in \mathcal M(\mu,\nu)} H(\pi|Q),
\end{equation}
where $\mathcal{M}(\mu, \nu)$ is defined in \eqref{eq:def-M} and $Q$ is the reference probability measure with log-density $-c$, that is, 
\[dQ/d(\mu\otimes \nu\otimes\rho) = \exp(-c(x,y,z)).\]
Before we present the dual problem and Sinkhorn's algorithm, we first list some assumptions on marginal distributions and the regularity of function $c$.
\begin{assumption}
\label{assu:mu_nu}
\begin{itemize}
    \item[(i)] $\mu$, $\nu$ are supported on compact sets;
    \item[(ii)] The probability measure $\nu\neq \delta_{y_0}$ for any $y_0\in \mathbb R$;
    \item[(iii)] $\int_{\mathbb R} x\mu(dx) = \int_{\mathbb R} y\nu(dy) = 0$.
\end{itemize}
\end{assumption}
We note that the second and the third conditions of Assumption \ref{assu:mu_nu} are assumptions without loss of generality, because if they are violated, a martingale transport between $\mu$ and $\nu$ either does not exist or exists only in a trivial form. Under Assumption \ref{assu:mu_nu}, we denote by $\mathcal X_0$ the bounded effective interval
\begin{equation*}
    \mathcal X_0 = \Big\{x\in \mathbb R,\ s.t.\ \mu([x,\infty))>0\ \textrm{and} \ \mu((-\infty,x])>0\Big\},
\end{equation*}
and similarly denote $\mathcal Y_0$. Then the outer bound of $\mathcal X_0$ is defined as $\overline{\mathcal X}=\sup_{x\in\mathcal X_0}x>0$ and $\underline{\mathcal X}=\inf_{x\in\mathcal X_0}x<0$, and the essential absolute bound of $\mu$ on $\mathcal X$ is denoted by $|\mathcal X|=\overline{\mathcal X}\vee(-\underline{\mathcal X})$. Similarly we define $\overline{\mathcal Y}$, $\underline{\mathcal Y}$ and $|\mathcal Y|$. 

\begin{assumption}
\label{assu:lipschitz_c_conti}
    For the potential function $c(x,y,z)$ of the reference probability measure $Q$, we assume that it is continuously differentiable and has essential bounds: 
    \begin{equation*}
    |\nabla c(x,y,z)|_{L^{\infty}(\mu\otimes \nu\otimes\rho)} \le L_c, \quad |c(x,y,z)|_{L^{\infty}(\mu\otimes \nu\otimes\rho)} \le |c|_{L^\infty},
    \end{equation*}
for some constants $L_c\ge0$ and $|c|_{L^\infty}\ge 0$.
\end{assumption}
Assumptions \ref{assu:mu_nu} (i) and \ref{assu:lipschitz_c_conti} will be used to derive the uniform $L^\infty$ bound for the dual coefficients, specifically, in Theorem \ref{thm:uniform_bound_conti}. The continuity of the function $c$ is used to deduce the continuity of the dual coefficients. We further assume that there exists a martingale optimal transport equivalent to $\mu\otimes\nu\otimes \rho$.

\begin{assumption}
\label{assu:exsitence_mg_transport_conti}
There exists $\bar \pi \in \mathcal M(\mu,\nu)$ such that $\bar \pi\sim \mu\otimes \nu\otimes \rho$, moreover, it takes the form $ d\bar \pi/d(\mu\otimes \nu\otimes \rho) = \exp(-p(x,y,z)) $, with a potential $p(x,y,z)$ bounded from both sides,  $\mu\otimes \nu\otimes \rho$-a.s., i.e., $|p|_{L^\infty(\mu\otimes \nu\otimes \rho)}<\infty$; 
\end{assumption}

Assumption~\ref{assu:exsitence_mg_transport_conti} is strictly stronger than the mere existence of a martingale transport with finite relative entropy $H(\pi\,|\,Q)<\infty$, because the finiteness of $H(\pi|Q)$ only yields  $\pi \ll \mu\otimes \nu\otimes \rho$ but not the boundedness of the log-density. As we will see, the presumed existence of such transport plan $\bar \pi$ plays a crucial role in our arguments to prove the quantitative convergence of Sinkhorn's algorithm, in particular, in order to prove Proposition \ref{prop:existence_p1_conti}. 

To introduce the dual problem of \eqref{eq:mg_transport}, we define the Lagrange function as
\begin{equation*}
\begin{aligned}
        \mathcal L(\pi,f,g, &h) = H(\pi |Q) + \int_{\mathcal X} f(x)\Big(\int_{\mathcal Y\times\mathcal Z} \pi(dx,dy,dz)- \mu(dx)\Big) \\
        & + \int_{\mathcal Y}g(y)\Big(\int_{\mathcal X\times\mathcal Z} \pi(dx,dy,dz) - \nu(dy)\Big) 
          +\int_{\mathcal X} h(x)  \Big(\int_{\mathcal Y\times\mathcal Z} (y-x)\pi(dx,dy,dz) \Big),
\end{aligned}
\end{equation*}
for any $f\in L^\infty(\mu)$, $g\in L^\infty(\nu)$, $h\in L^\infty(\mu)$ and any positive measure $\pi$.
% \begin{remark}
%     It is noticed that a necessary condition for $H(\pi|Q)<+\infty$ is that $\pi\ll\mu\otimes \nu$, which further implies in iterations, only the value of $f,h$ (resp. $g$) on the support of $\mu$ (resp. $\nu$) is concerned. 
% \end{remark}
% We see from the lines above that $\pi\ll \mu\otimes\nu$ is a necessary condition to derive the finiteness of the optimizing objective or Lagrange function.
We further define $\mathcal G(f,g,h)$ to be the infimum of $\mathcal L(\pi,f,g,h)$ over positive measures $\pi$ given $(f,g,h)$:
\begin{equation*}
    \mathcal G(f,g,h) := \inf_{\pi\in \mathcal M^+(\mathcal X\times \mathcal Y\times \mathcal Z)} \mathcal L (\pi,f,g,h).
\end{equation*}
 By direct computation, we have
\[\mathcal G(f,g,h)  = - Z(f,g,h) - \int_{\mathcal X} f(x)\mu(dx) - \int_{\mathcal Y} g(y) \nu(dy),\]
where $Z(f,g,h):=\pi(f,g,h)(\mathcal X\times \mathcal Y\times \mathcal Z)$ is the total mass of a positive measure $\pi(f,g,h)$ defined by:
%the mapping $\pi:L^0(\mathcal X)\times L^0(\mathcal Y) \times (L^0(\mathcal X))^d \mapsto \mathcal M^+(\mathcal X\times \mathcal Y)$:
\begin{equation}
\label{eq:definition_of_induced_prob}
\frac{d\pi(f,g,h)}{d(\mu\otimes \nu\otimes \rho)} = \exp\big(-c(x,y,z)-f(x)-g(y) - h(x) (y-x)\big). 
\end{equation}
The dual problem to \eqref{eq:mg_transport} reads
\begin{equation}
\label{eq:dual_problem}
    \sup_{f\in L^\infty(\mu),g\in L^\infty(\nu),h\in L^\infty(\mu)}\mathcal G(f,g,h).
\end{equation}
To solve the dual problem, we turn to Sinkhorn's algorithm in \eqref{eq:sinkhorn_alg_min_conti}. In the $n$-th iteration, we update $h^n \to \tilde {h}^{n+1}$ using the first-order condition for optimality:
\begin{equation}
\label{eq:h_equality_conitnuous}
    \tilde{h}^{n+1} := h
\end{equation}
such that
\begin{equation*}
    \int_{\mathcal Y\times\mathcal Z} (y-x)\exp(-c(x,y,z)- f^n(x)-{g}^{n}(y)-h(x)(y-x))(\nu\otimes\rho)(dy,dz)= 0,
\end{equation*}
and again by the first-order condition of optimality for $\tilde{f}^{n+1},\tilde{g}^{n+1}$ , functions $\tilde f^{n+1},\tilde g^{n+1}$ are updated as
\begin{equation}
\label{eq:update_formula_fg}
\begin{aligned}
     & \tilde{f}^{n+1}(x) := \log \Big(\int_{\mathcal Y\times\mathcal Z}\exp\big(-c(x,y,z) - g^{n}(y) - \tilde{h}^{n+1}(x)(y-x)\big)(\nu\otimes\rho)(dy,dz)\Big)  , \\ % \sum_i \mu_i \log \left(\sum_j \exp(-u_{ij} - g_j^n - h_i^n(y_j-x_i))\right)   
     & \tilde{g}^{n+1}(y) := \log \Big(\int_{\mathcal X\times\mathcal Z} \exp\big(-c(x,y,z) -  \tilde{f}^{n+1}(x) -  \tilde{h}^{n+1}(x)(y-x)\big)(\mu\otimes\rho)(dx,dz)\Big). %\sum_j \nu_j \log \left(\sum_i \exp(-u_{ij} - f_i^{n+1} - h_i^{n+1}(y_j-x_i))\right)
\end{aligned}
\end{equation}
In particular, the existence and the measurability of ${\tilde h}^{n+1}$ will be ensured by Proposition \ref{prop:well_defined_fgh}. Finally we normalize $(\tilde{f}^{n+1},\tilde{g}^{n+1},\tilde{h}^{n+1}) $ such that
\begin{equation}
\label{eq:normalization_iter_n}
    (f^{n+1},g^{n+1},h^{n+1}) := \big(\tilde{f}^{n+1}+ \lambda_{\tilde{g}^{n+1}}-\lambda_{\tilde{h}^{n+1}} x,\tilde{g}^{n+1}-\lambda_{\tilde{g}^{n+1}}+\lambda_{\tilde{h}^{n+1}} y,\tilde h^{n+1}-\lambda_{\tilde{h}^{n+1}} \big),
\end{equation}
where $\lambda_{\tilde{g}^{n+1}} = \int_{\mathcal Y}\tilde g^{n+1}(y)\nu(dy)$, $\lambda_{\tilde{h}^{n+1}} = \int_{\mathcal X} \tilde h^{n+1}(x)\mu(dx)$ so that
\begin{equation}\label{eq:normalization}
    \int_{\mathcal Y} g^{n+1}(y)\nu(dy)=\int_{\mathcal X} h^{n+1}(x)\mu(dx)=0.
\end{equation}
\begin{remark}
    It is important to observe that in both the primal problem~\eqref{eq:mg_transport} and the dual problem~\eqref{eq:dual_problem}, only the values of the functions \( f \), \( g \), and \( h \) on the supports of \( \mu \), \( \nu \), and \( \mu \) respectively, are relevant. This stems from the fact that for the functional \( \mathcal{L}(\pi, f, g, h) \) to be well defined with finite relative entropy \( H(\pi | Q) < +\infty \), it is necessary that \( \pi \ll \mu \otimes \nu \otimes \rho \). Nevertheless, for analytical convenience—particularly to facilitate the discussion of continuity during the iterative updates—we define the updated triplet \( (\tilde{f}^{n+1}, \tilde{g}^{n+1}, \tilde{h}^{n+1}) \) over the entire domain rather than restricting to the supports, using the continuous formulations~\eqref{eq:h_equality_conitnuous} and~\eqref{eq:update_formula_fg}.
\end{remark}
\begin{remark}
    The values of  $\pi(f,g,h)$ in \eqref{eq:definition_of_induced_prob} and $\mathcal G(f,g,h)$ remain unchanged under the following two types of transforms
    \begin{equation}
    \label{eq:invariant_transform}
    \left\{\begin{array}{c}
        f \mapsto f+c_1 \\
         g\mapsto g- c_1   \\
    \end{array}\right.
    \text{\quad or\quad } 
    \left\{\begin{array}{c}
        f\mapsto f-c_2  x  \\
         g\mapsto g+c_2 y  \\
         h\mapsto h-c_2 \\
    \end{array}\right.,
\end{equation}
for any constant $c_1,c_2\in \mathbb R$. In particular, the normalization \eqref{eq:normalization_iter_n} does not affect the optimality and the subsequent Sinkhorn iteration steps. 
\end{remark}

% Following from Assumptions \ref{assu:exsitence_mg_transport_conti}, we will use the following proposition for the essential boundary property of $\mu$, $\nu$.

\begin{assumption}
\label{assu:measuable_init_bdd}
    % \textcolor{red}{(i) The measure $\mu$ and $\nu$ can be decomposed into absolutely continuous and singular parts $\mu=\mu_c+\mu_s$ and $\nu =\nu_c +\nu_s$ with respect to the Lebesgue measure $\mu,\nu \ll \mathbb R$, the support of $\mu$ and $\nu$ is simply connected, and the density of $\mu_c$ and $\nu_c$ are uniformly lower bounded away from $0$
    % \begin{equation*}
    %     \log \mu_c \ge -c_\mu, \ \mu_c{\text-a.s}\quad  {\text and} \quad \log \nu_c \ge -c_\nu, \ \nu_c{\text-a.s};
    % \end{equation*}}
     The initial input of Sinkhorn's algorithm $(f^0,g^0,h^0)$ is measurable and bounded in the sense that $f^0\in  L^\infty(\mu), g^0 \in L^\infty(\nu)$, $h^0\in L^\infty(\mu)$. 
\end{assumption}

\begin{proposition}
\label{prop:well_defined_fgh}
    Under Assumptions \ref{assu:mu_nu}, \ref{assu:lipschitz_c_conti} and \ref{assu:measuable_init_bdd},
    the iteration steps $(\tilde f^{n+1}, \tilde g^{n+1})_{n\ge 1}$ by Sinkhorn's iterations are well-defined on $\mathbb R$ in \eqref{eq:update_formula_fg}, and $(\tilde h^{n+1})_{n\ge1}$ are well defined on $\mathcal X_0$ in \eqref{eq:h_equality_conitnuous}.
    In particular, the normalized functions $(f^n,g^n,h^n)_{n\ge 1}$ in \eqref{eq:normalization_iter_n} are continuously differentiable and bounded, i.e., $f^n,h^n\in L^\infty(\mu)\cap  C^1(\mathcal X_0)$ and $g^n \in L^\infty(\nu)\cap C^1(\mathcal Y_0)$. Moreover, if  $g^n\in L^\infty(\nu)$, then there exists a constant $c_h:=c_h(|c|_{L^\infty},|g^n|_{L^\infty(\nu)},\nu)>0$ such that $|\tilde h^{n+1}|_{L^\infty(\mu)}\le c_h$.
\end{proposition}

The proof of Proposition \ref{prop:well_defined_fgh} can be found in Appendix \ref{sec:well_defined_fgh}. Our proof for the main convergence result for Sinkhorn's algorithm relies heavily on the following construction of a probability measure. 

\begin{proposition}
\label{prop:existence_p1_conti}
    Under Assumption \ref{assu:exsitence_mg_transport_conti}, there exists a transport $\hat \pi \in \Pi(\mu,\nu)$ and a constant $\epsilon_p>0$ such that 
    \begin{equation}
    \label{eq:pi1_property_semi_mg}
    \int_{\mathcal Y\times\mathcal Z} (y-x) d \hat \pi_{y,z|x} \left\{
    \begin{array}{l}
        > \ \ \epsilon_p, \text{\quad for\ any\ }  x\ge 0;   \\
        < -\epsilon_p,  \text{\quad for\ any\ } x< 0.
    \end{array}
    \right.
    \end{equation}
    Meanwhile, the quotient density $|\log d\hat \pi/d(\mu\otimes \nu\otimes\rho)|$ is bounded and the relative entropy to the reference measure satisfies $H(\hat \pi|Q)<+\infty$.
\end{proposition}

\begin{remark}
\label{rmk:relation_nutz}
Proposition~\ref{prop:existence_p1_conti} is closely connected to Lemma~3.2 in \cite{nutz2024martingale}, which proves the existence of a martingale measure $\hat{\pi}$ satisfying \eqref{eq:pi1_property_semi_mg} and  $H(\hat{\pi}\,|\,Q)<+\infty$. The main difference lies in the additional boundedness condition required in Assumption~\ref{assu:exsitence_mg_transport_conti}, which assumes the quotient density $\log (d\bar{\pi}/d(\mu\otimes\nu\otimes\rho))$ to be bounded. While this condition is not required for the existence result of \cite{nutz2024martingale}, it is needed to construct a coupling $\hat \pi$ such that  $\log (d\hat{\pi}/d(\mu\otimes\nu\otimes\rho))$ is bounded, which plays a crucial role in our quantitative analysis on the Sinkhorn iteration, ensuring uniform control of the dual potentials in Theorem~\ref{thm:uniform_bound_conti}. 
\end{remark}

The Proposition \ref{prop:existence_p1_conti} can be proved by perturbing the martingale transport $\bar \pi$ whose existence is assumed in Assumption \ref{assu:exsitence_mg_transport_conti}. Note that the choice of $\bar \pi$ will affect the value of $\epsilon_p$, which further affects the bound for the convergence rate of Sinkhorn's algorithm. A more precise bound for convergence rate should traverse all possible choices of $\bar \pi \in \mathcal M(\mu,\nu)$, and the corresponding $\hat \pi$. Now we present the convergence results of Sinkhorn's algorithm.

\begin{theorem}
\label{thm:uniform_bound_conti}
    Under Assumptions \ref{assu:mu_nu}, \ref{assu:lipschitz_c_conti}, \ref{assu:exsitence_mg_transport_conti} and \ref{assu:measuable_init_bdd}, for dual coefficients triples $(f^n,g^n,h^n)_{n\ge0}$ updated according to Sinkhorn's algorithm, \eqref{eq:h_equality_conitnuous} and \eqref{eq:update_formula_fg}, we can find some constants $C_f,C_g,C_h >0$ such that 
    \begin{equation*}
    \begin{aligned}
                \sup_{n\ge 1} |h^n(x)|_{L^\infty(\mu)} \le C_h, \quad       \sup_{n\ge 1} |f^n(x)|_{L^\infty(\mu)} \le C_f,\quad \text{and} \quad  \sup_{n\ge 1} |g^{n}(y)|_{L^\infty( \nu)} \le C_g,
    \end{aligned}
    \end{equation*}
    where these constants are dependent on $|\mathcal X|, |\mathcal Y|, |c|_{L^\infty}, |p|_{L^\infty}, \epsilon_p$.
\end{theorem}

The proof of Proposition \ref{prop:existence_p1_conti} and Theorem \ref{thm:uniform_bound_conti} can be found in Section \ref{proof:uniform_bound}. Subsequently, we observe that the induced probability satisfies $\sup_{n\in \mathbb N^+}|\log d\pi(f^n,g^n,h^n)/d(\mu\otimes\nu\otimes\rho)|< \infty$. And this fact enables us to show the exponential convergence of Sinkhorn's algorithm. 

\begin{theorem}
\label{thm:existence_minimizer_linear_convergence} Under Assumptions \ref{assu:lipschitz_c_conti}, \ref{assu:exsitence_mg_transport_conti} and \ref{assu:measuable_init_bdd}, we have
\begin{enumerate}
    \item[{\rm (1)}] Sinkhorn's algorithm for the dual problem, \eqref{eq:h_equality_conitnuous} and \eqref{eq:update_formula_fg}, converges exponentially in the following sense:
    \begin{equation}
    \label{eq:thm_exp_cvgz}
        \mathcal G^{*} - \mathcal{G}(f^{n},g^{n},h^{n}) \le \big( 1-\frac{1}{C_s+1} \big)^n \big(\mathcal G^{*} - \mathcal{G}(f^{0},g^{0},h^{0})\big),
    \end{equation}
    for some $C_s >0$, depending on the parameters $|\mathcal X|_{\infty},|\mathcal Y|_{\infty}, C_f,C_g,C_h$ (see its detailed definition in \eqref{eq:C_s}).
    And where $\mathcal G^{*} = \sup_{f\in L^\infty(\mu),g\in L^\infty(\nu),h\in L^\infty(\mu)}\mathcal{G}(f,g,h)$;

    \item[{\rm (2)}] There exist a triple $(f^*,g^*,h^*)$ such that 
    $f^n\to f^*$ in $L^\infty(\mu)$, $g^n\to g^*$ in $L^\infty(\nu)$ and $h^n\to h^*$ in $L^\infty(\mu)$, which attains the maximum of the dual problem
    \eqref{eq:dual_problem}, that is,
    \begin{equation}
    \label{eq:Gstar-Gfuncstar}
     \mathcal G(f^*,g^*,h^*) =\mathcal G^{*}.
    \end{equation}
    The optimal dual coefficients triple $(f^*,g^*,h^*)$ is the unique maximizer of $\mathcal G$ that satisfies the normalization conditions $\int_{\mathcal Y}gd\nu = \int_{\mathcal X}hd\mu = 0$;
    
    \item[{\rm (3)}]  The dual functions $f^n,g^n,h^n$
    converge exponentially to $f^*,g^*,h^*$ in $L^2$ norm, i.e., 
    \begin{equation}
    \label{eq:L2convergence_formula}
        |f^n - f^*|_{L^\infty(\mu)}^2 + |g^n - g^*|_{L^\infty(\nu)}^2 + |h^n - h^*|_{L^2(\mu)}^2 \le C_0 \big( 1-\frac{1}{C_s+1} \big)^n,
    \end{equation}
    for some constant $C_0 \ge 0$, where the constant $C_s$ is the same as in \eqref{eq:thm_exp_cvgz}.
    
    \item[{\rm (4)}] The probability measure induced by the limit triple $(f^*,g^*,h^*)$:
    \begin{equation*}
        \frac{d\pi^*}{d(\mu\otimes \nu\otimes\rho)} = \exp\big(-c(x,y,z) - f^*(x) - g^*(y) - h^*(x)(y-x)\big),
    \end{equation*}
    is the unique minimizer $\pi^* \in \mathcal M(\mu,\nu)$ to the primal problem \eqref{eq:mg_transport}, i.e.,
    \begin{equation*}
        \inf_{\pi \in \mathcal M(\mu,\nu)} H(\pi |Q) = \mathcal G^{*}.
    \end{equation*}
\end{enumerate}
\end{theorem}
\begin{remark}
    Theorem~\ref{thm:existence_minimizer_linear_convergence} provides a lower bound on the convergence rate of the proposed Sinkhorn algorithm. However, this bound is not given in an explicit form, and depends on several problem-dependent constants whose behavior is not fully characterized. In Section~\ref{sec:applications}, we conduct numerical experiments to complement this theoretical result. These experiments aim to shed light on how some of these constants influence both the theoretical lower bound obtained above and the practical convergence speed.
\end{remark}

\section{Applications}
\label{sec:applications}

\subsection{Numerical Test on Convergence Rate}
To further understand the convergence behavior of the proposed Sinkhorn-type algorithm, we conduct a numerical study. In the proof of Theorem \ref{thm:existence_minimizer_linear_convergence}, we established a lower bound for the convergence rate, but no explicit discussion was provided regarding the actual speed of convergence in practice. Therefore, we aim to analyze some key factors that might influence the convergence rate.

\paragraph{Coefficient of Entropy regularization} We first investigate the influence of the coefficient in the entropy regularization. Specifically, we modify the reference Gibbs measure $Q$ by introducing a scaling parameter $\sigma>0$ as follows:
\[
\frac{d Q_\sigma}{d(\mu \otimes \nu \otimes \rho)}(x, y, z)
:= \frac{1}{Z_\sigma}\exp\!\left(-\frac{c(x, y, z)}{\sigma}\right),
\quad 
Z_\sigma
:= \int_{\mathcal X\times\mathcal Y\times\mathcal Z}
\exp\!\left(-\frac{c(x, y, z)}{\sigma}\right)
\, d(\mu\otimes\nu\otimes\rho),
\]
so that $Q_\sigma$ is a probability measure on $\mathcal X\times\mathcal Y\times\mathcal Z$. 
Under this definition, the entropic martingale optimal transport problem \eqref{eq:primal_intro} can be equivalently written as
\begin{equation*}
    \inf_{\pi\in \mathcal M(\mu,\nu)} H(\pi|Q_\sigma)
    = \inf_{\pi \in \mathcal M(\mu,\nu)} 
    \Big[
        H(\pi|\mu \otimes \nu\otimes \rho)
        + \frac{1}{\sigma} 
        \int_{\mathcal X\times \mathcal Y\times \mathcal Z}
        c(x,y,z)\,d\pi + \log Z_\sigma
    \Big],
\end{equation*}
where the coefficient $\sigma$ controls the relative weight between the entropy term and the transport cost.
% Specifically, we modify the Gibbs measure as $\frac{d Q}{d\mu \otimes \nu \otimes \rho}(x, y, z) \propto \exp\left(-\frac{c(x, y, z)}{\sigma}\right)$, where \(\sigma > 0\) acts as a scaling parameter. Under this modification, the primal problem \eqref{eq:primal_intro} can be rewritten as
% \begin{equation*}
%     \inf_{\pi\in \mathcal M(\mu,\nu)} H(\pi|Q) = \inf_{\pi \in \mathcal M(\mu,\nu)} H(\pi|\mu \otimes \nu\otimes \rho) + \frac{1}{\sigma} \int_{\mathcal X\times \mathcal Y\times \mathcal Z} c(x,y,z)\,d\pi,
% \end{equation*}
% where the coefficient $\sigma$ adjusts the relative weight between the entropy term and the transport cost.
In our experiment, we consider the cost function \( c(x,y,z) = (x - y)^2 + (y - z)^2 \), and test the algorithm under varying values of \(\sigma = 0.2, 1, 5\), while keeping all other components fixed, including the marginal distributions and the initialization.

Figure \ref{image:convergence_rate_of_different_coefficient_of_H} shows how the convergence rate varies with different values of 
\(\sigma\). As 
\(\sigma\) increases, the observed lower bound on the convergence rate also increases. Nevertheless, the central portions of the convergence curves exhibit nearly identical linear decay on a logarithmic scale, regardless of the value of 
\(\sigma\). This suggests that, in this example, the overall convergence speed is largely insensitive to the strength of the entropy regularization.

\begin{remark}
In Figure~\ref{image:convergence_rate_of_different_coefficient_of_H}, we display the relative error $H_n - H_*$ on a logarithmic scale. 
For each fixed $\sigma>0$, the functional value at iteration $n$ is defined by
\[
H_n := H\!\left(\pi^{\,n}\middle | Q_\sigma\right) - \log Z_\sigma
= H\!\left(\pi^{\,n}\middle | \mu\otimes\nu\otimes\rho\right)
+ \frac{1}{\sigma}\!\int_{\mathcal X\times\mathcal Y\times\mathcal Z} c(x,y,z)\,d\pi^{\,n},
\]
where $\pi^{\,n}$ denotes the induced coupling at iteration $n$ (cf.~\eqref{eq:definition_of_induced_prob}).  
The limiting value
\[
H_* := \lim_{n\to\infty} H_n
\]
corresponds to the minimal entropic cost of the regularized problem. The figure therefore plots the sequence of relative errors $H_n - H_*$ to illustrate the convergence behavior of the algorithm under different regularization coefficients~$\sigma$.
\end{remark}

\begin{figure}[htbp]
    \centering
    \begin{minipage}[h]{0.45\linewidth}
        \centering
        \includegraphics[width=\textwidth]{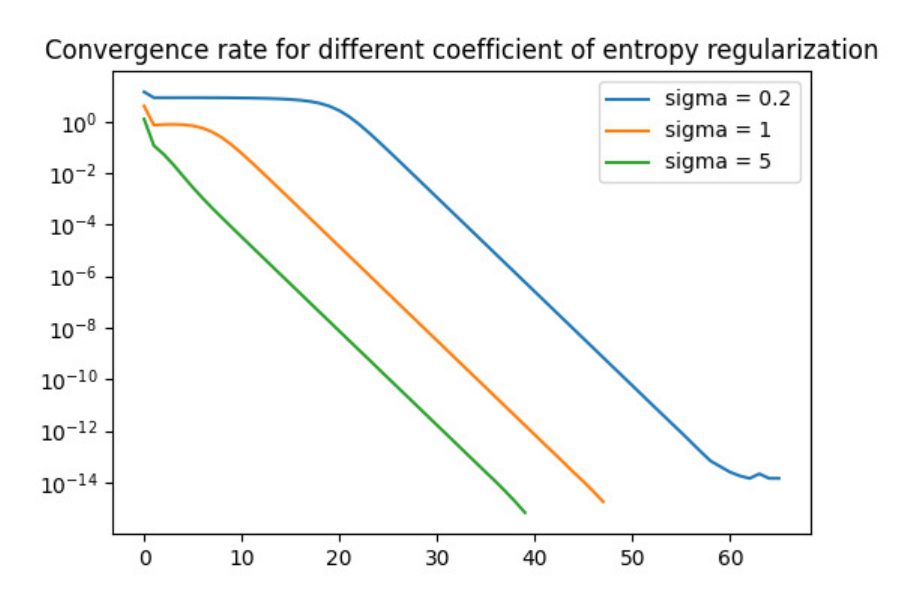}
        \caption{Convergence rate for different coefficient of Entropy regulization.}
        \label{image:convergence_rate_of_different_coefficient_of_H}
    \end{minipage}
    \hfill
    \begin{minipage}[h]{0.45\linewidth}
        \centering
        \includegraphics[width=\textwidth]{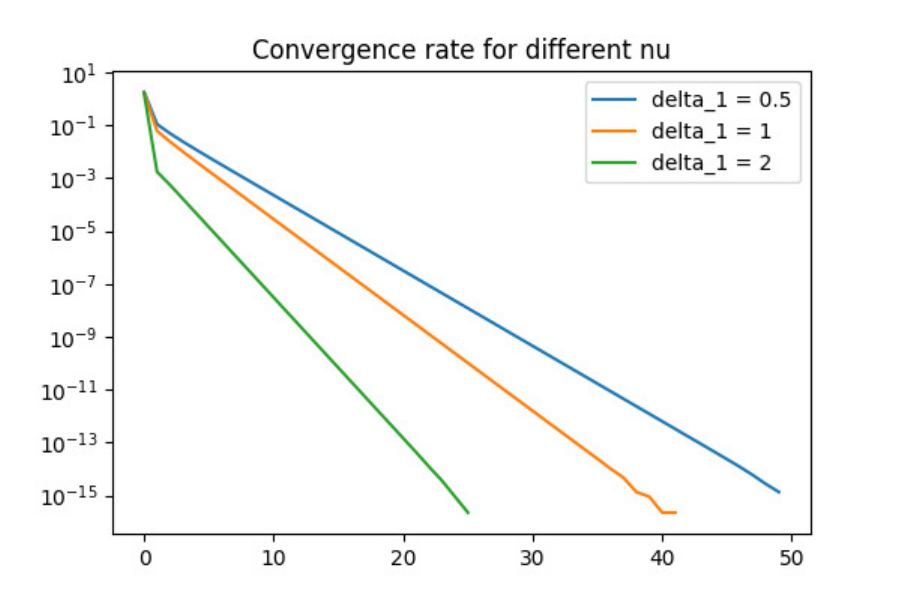}
        \caption{Convergence rate for different level of convex order in $\nu$.}
        \label{image:Convergence_rate_for_different_nu}
    \end{minipage}
\centering
\end{figure}

\paragraph{Dependence on $\nu$}
In this experiment, we fix $\mu$ as a uniform distribution supported on a finite set of grid points $X=\{x_1, \cdots, x_n\}$ such that $x_1< \cdots< x_n$. The marginal $\nu$ is also taken as a uniform distribution, but with two additional points added to the ends of the support, that is, supported on $Y:=\{x_0,X, x_{n+1}\}$ such that $x_0<x_1$ and $x_n<x_{n+1}$. We let $x_{i+1}-x_{i}=\delta_2$ for $i=1, \cdots, n-1$, while the two outermost points are placed symmetrically at a distance $\delta_1$ from the nearest interior points, that is, $x_1-x_0 = x_{n+1} -x_n = \delta_1$. We use $\delta_1$ to quantify the deviation of $\nu$ from the support of $\mu$, and examine how it influences the convergence behavior of Sinkhorn's algorithm.

According to Theorem \ref{thm:uniform_bound_conti} and \ref{thm:existence_minimizer_linear_convergence}, the lower bound on the convergence rate depends on a regularity parameter \(\varepsilon_p\). In particular, the larger $\varepsilon_p$ we have, the larger the lower bound of the convergence of Sinkhorn's algorithm is.  In our setting,  we observe that \(\varepsilon_p\) is increasing with respect to  \(\delta_1\) (see more detail on the construction of $\varepsilon_p$ in the proof of Proposition \ref{prop:existence_p1_conti}). This leads to the theoretical prediction that a larger \(\delta_1\) may lead to a quicker convergence. Figure \ref{image:Convergence_rate_for_different_nu} shows that the convergence speed of the Sinkhorn algorithm improves with larger \(\delta_1\), consistent with the theoretical expectation.

\begin{remark}
An alternative and perhaps more intuitive explanation for the observed convergence rate may be  a  consequence of enhanced convex ordering between \(\mu\) and \(\nu\). Denote by \(\nu'\) the new measure when the outermost support points of \(\nu\) are pushed further apart. Note that \(\nu'\) is larger than $\nu$ in convex order. This perspective aligns well with the Jensen inequality, which characterizes convex order via \(\mathbb{E}_{\nu'}[\varphi] \geq \mathbb{E}_\nu[\varphi]\) for all convex functions \(\varphi\). 
\end{remark}

\subsection{Calibration of Stochastic Volatility Models via EMOT}
\subsubsection{Background}
\label{sec:calibration}
The study of stochastic volatility models (SVMs) dates back to the 1970s, motivated by the empirical observation that the log-returns of SPX exhibit pronounced peaks and heavy tails, deviating significantly from the normal distribution. Seminal works by Hull and White~\cite{hull1987pricing}, Heston~\cite{heston1993closed}, Chesney and Scott~\cite{chesney1989pricing}, and Engle~\cite{engle1982autoregressive} established the theoretical foundations of SVMs and highlighted their practical importance. A systematic introduction and mathematical formulation of these models can be found in~\cite{gatheral2011volatility}, where volatility is modeled as a stochastic process to capture stylized facts such as the volatility smile.

Consider the following naive SVM, introduced in~\cite{henry2019martingale}, defined over the time interval $[0,T]$:
\begin{equation}
\label{eq:SPX/VIX_eqn1}
\begin{aligned}
    dS_t &= S_t a_t\, dW^0_t,\quad d\langle B^0,W^0\rangle_t = \rho\, dt,\\
    da_t &= b(a_t)\, dt + \sigma(a_t)\, dB^0_t,\quad \text{for } t_1 \le t \le t_2,
\end{aligned} \quad \mathbb{P}^0\text{-a.s.}
\end{equation}
Here, \( W^0 \) and \( B^0 \) are $\mathbb{P}^0$-Brownian motions with correlation \( \rho \in [-1,1] \). Traditional calibration of the model to market data involves selecting the functions \( b(\cdot) \), \( \sigma(\cdot) \), and the parameter \( \rho \) so that the model-implied option prices best match observed market prices.

In contrast, the approach in~\cite{henry2019martingale} proposes an alternative calibration method based on the Martingale Schrödinger Bridge (MSB). The idea is to start from an initial guess of \( b(\cdot) \), \( \sigma(\cdot) \), and \( \rho \), which induces a reference measure \( \mathbb{P}^0 \) on the path space. The goal is then to solve the following MSB problem:
\begin{equation}
\label{eq:intro_schrodinger_bridge_conti}
    \inf_{\mathbb P}  H(\mathbb P | \mathbb P^0) \quad \text{subject to} \quad 
    \begin{cases}
        S_t \text{ is a } \mathbb{P}\text{-martingale}, \\
        \mathbb{P} \circ S_{t_i}^{-1} = m_{t_i}, \quad i=1,2,\dots,I,
    \end{cases}
\end{equation}
where \( (m_{t_i})_{i=1}^I \) are the market-implied marginal distributions of the asset at future times \( (t_i)_{i=1}^I \). These marginals can be inferred from option prices. For instance, in the case of European options, we have the relations
\begin{equation*}
    C_{\text{market}}(t_i,K) = \mathbb{E}^{\mathbb{P}}[(S_{t_i} - K)^+], \quad 
    P_{\text{market}}(t_i,K) = \mathbb{E}^{\mathbb{P}}[(K - S_{t_i})^+],
\end{equation*}
which, assuming sufficient smoothness, yield the densities
\[
    m_{t_i}(K) = \frac{\partial^2 C_{\text{market}}(t_i,K)}{\partial K^2} 
    \quad \text{or} \quad 
    m_{t_i}(K) = -\frac{\partial^2 P_{\text{market}}(t_i,K)}{\partial K^2}.
\]
provided that the function $K \mapsto C_{\text{market}}(t_i,K)$ or $K \mapsto P_{\text{market}}(t_i,K)$ is twice continuously differentiable. Once the MSB problem is solved, the resulting measure \( \mathbb{P} \) serves as the calibrated model. In this sense, the MSB framework performs a ``non-parametric fine-tuning" of the initial SVM, ensuring consistency with market-observed marginals while remaining close (in relative entropy) to the reference dynamics.

To make the problem computationally tractable, one may discretize the MSB problem~\eqref{eq:intro_schrodinger_bridge_conti} into a multi-period entropy-regularized martingale optimal transport (EMOT) problem:
\begin{equation}
\label{eq:discrete_MOT_SVM_prob}
    \inf_{\pi \in \mathcal{P}\left( \prod_{i=0}^I X^i \times \prod_{i=0}^{I-1} V^i \right)} 
    H(\pi | Q),
\end{equation}
on the state space \( \prod_{i=0}^I X^i \times \prod_{i=0}^{I-1} V^i \), subject to the marginal and martingale constraints:
\begin{equation*}
    \pi^i = m_{t_i},\quad \int (x^{i+1} - x^i)\, d\pi^{i+1|i} = 0, \quad \text{for } i=0,\dots,I-1.
\end{equation*}
Thanks to the Markov structure in~\eqref{eq:SPX/VIX_eqn1}, the joint measure \( \pi \) admits the decomposition
\[
\pi = \pi^0 \cdot \prod_{i=0}^{I-1} \pi^{i+1|i}, \quad Q = Q^0 \cdot \prod_{i=0}^{I-1} Q^{i+1|i},
\]
so that the relative entropy satisfies the chain rule:
\begin{equation}
\label{eq:relative_entropy_chain}
H(\pi | Q) = H(\pi^0 | Q^0) + \sum_{i=0}^{I-1} \mathbb{E}_{x^i\sim \pi^i} \left[ H(\pi^{i+1|i}(\cdot \mid x^i) | Q^{i+1|i}(\cdot \mid x^i)) \right].
\end{equation}
Since the marginals \( \pi^i \) are fixed to match \( m_{t_i} \), the term \( H(\pi^i | Q^i) \) becomes constant and can be ignored for optimization. As a result, the global EMOT problem~\eqref{eq:discrete_MOT_SVM_prob} can be decoupled into \( I \) one-step subproblems of the form:
\begin{equation}
\label{eq:discrete_MOT_SVM_subprob_clean}
\inf_{\pi^{i,i+1} \in \mathcal{P}(X^i \times X^{i+1} \times V^i)} 
H(\pi^{i,i+1} | Q^{i,i+1})
\end{equation}
subject to
\begin{equation}
\pi^i = m_{t_i}, \quad \pi^{i+1} = m_{t_{i+1}}, \quad 
\int (x^{i+1} - x^i)\, d\pi^{i+1|i} = 0,
\end{equation}
on the state space \( X^i \times X^{i+1} \times V^i \), where \( (S_{t_i}, S_{t_{i+1}}, a_{t_i}) \) resides.

Each of these subproblems is an instance of the entropy-regularized martingale optimal transport problem and can be efficiently solved using the Sinkhorn algorithm. In view of Theorem~\ref{thm:existence_minimizer_linear_convergence}, we can establish a dual formulation and expect exponential convergence. For clarity and numerical illustration, we focus on the one-period case in the next subsection.

\begin{remark}
The Schrödinger Bridge problem, first questioned by Schrödinger in \cite{schrodinger1931umkehrung}, aims to find the most likely stochastic evolution between two probability distributions with prior probability. The Martingale Schrödinger Bridge problem adds the constraint of martingale property and was first introduced in \cite{ henry2019martingale}, regarded as the entropic approximation of the martingale optimal transport problem \cite{beiglbock2013model,de2018entropic,galichon2014stochastic} and recognized as an approach to achieve perfect calibration by Vanilla option prices based on customized models \cite{henry2019martingale}, and compatible with models like stochastic volatility models (SVM). It is also noteworthy that the Martingale Schrödinger Bridge based approach offers a novel model that has the potential to solve the problem of joint calibration of SPX/VIX indices \cite{guyon2024dispersion}, where the classical SVM models can find their difficulties in capturing the prices of options on SPX/VIX simultaneously \cite{jacquier2018vix, song2012tale}. 
\end{remark}

\subsubsection{Numerical Implementation}
In this paper, we conduct numerical experiments of Sinkhorn's algorithm in solving the EMOT problem, and the discrete-space version is listed in Algorithm \ref{alg:sinkhorn_alg}. Here, the product $\mu_i \nu_j \rho_k$ represents the weight of the reference measure \( Q \) on the discrete grid point \( (x_i, y_j, z_k) \), assuming independence among the marginals.
\begin{algorithm}
\caption{Sinkhorn's Algorithm (discrete space case)}
\label{alg:sinkhorn_alg}
\KwIn{State space vectors $x=(x_i)_{i=1}^{N}\in\mathbb{R}^{N}, y=(y_j)_{j=1}^{M}\in\mathbb{R}^{M}, z=(z_k)_{k=1}^{L}\in\mathbb{R}^{L}$; marginal distributions $\mu, \nu, \rho$; initial dual functions $(f^0,g^0,h^0)\in \mathbb{R}^{N} \times \mathbb{R}^{ M} \times \mathbb{R}^{ N}$; cost tensor: $c_{ijk} := c(x_i, y_j, z_k)$ precomputed on the grid $(x_i, y_j, z_k)$; reference measure $Q$ with tensor entries proportional to $\exp(-c_{ijk}) \cdot \mu_i \nu_j \rho_k$; total number of iterations $T$.}
\For{$t=0,1,\ldots,T-1$}{
Solve $\forall i \in [N],\ \tilde{h}^{t+1}_i\leftarrow \tilde{h}_i$ \ s.t. $\sum_{j,k} (y_j-x_i)\exp(-c_{ijk}- f^t_i-g^t_j-\tilde{h}_i(y_j-x_i)) \nu_j\rho_k=0$, \\
$\forall i \in [N],\ \tilde{f}^{t+1}_i \leftarrow  \log \Big(\sum_{j,k}\exp\big(-c_{ijk} - g^{t}_{j} - \tilde{h}^{t+1}_{i}(y_j-x_i)\big)\nu_j\rho_k\Big)$,\\
$\forall j \in [M],\ \tilde{g}^{t+1}_j \leftarrow \log \Big(\sum_{i,k} \exp\big(-c_{ijk} -  \tilde{f}^{t+1}_i -  \tilde{h}^{t+1}_i(y_j-x_i)\big)\mu_i\rho_k\Big)$, \\
$(f^{t+1},g^{t+1},h^{t+1}) \leftarrow \big(\tilde{f}^{t+1}+\sum_{j}\tilde g_j^{t+1}\nu_j-x\sum_i\tilde h_i^{t+1}\mu_i,\tilde{g}^{t+1}-\sum_{j}\tilde g^{t+1}_j\nu_j+y\sum_i\tilde h_i^{t+1}\mu_i,\tilde h^{t+1}-\sum_i\tilde h_i^{t+1}\mu_i\big)$;
}
\KwOut{Optimized dual triples $(f^T,g^T,h^T) \in \mathbb{R}^{N} \times \mathbb{R}^{M} \times \mathbb{R}^{N}$; approximated solution of the EMOT problem in the discrete setting given by $\pi_{ijk} = \exp\big(-c_{ijk} - f^T_i - g^T_j - h^T_i \cdot (y_j - x_i)\big) \cdot \mu_i \nu_j \rho_k
\quad \text{for all } i \in [N], j \in [M], k \in [L]$.}
\end{algorithm}
\paragraph{Problem Setting}
% How to generate data, reference model, what is the approximated market data, initial of the f,g,h, the meaning of the parameters (probably use a table to contain all these paras).
Based on the content mentioned in Section 3.1, we generate market data using the classic SVM model: a combination of the Heston model and white noise. The Heston model is the most famous and popular among SVM models. We discretize the Heston process for Monte Carlo simulation and provide the joint distribution of stock prices at times $t_1$ and $t_2$, as well as the volatility of stock prices at time $t_1$. We present the Heston model
\begin{equation}
\label{eq:Heston}
\begin{aligned}
d S_t & = \sqrt{v_t} S_t d W_t \\
d v_t & =  -\lambda(v_t - \bar{v}) dt + \eta \sqrt{v_t} d Z_t,
\end{aligned}
\end{equation}
where $\lambda$ is the speed of reversion of $V_t$ to its long-term mean $\bar{v}$. Additionally, to ensure that the volatility remains strictly positive, we assume the Feller condition \cite{karlin1981second}, that is $\frac{2\lambda \bar{v}}{\eta^2} >1$.

During the numerical experiments, we aim to convert the simulated data generated by the Heston model into a three-dimensional joint probability distribution. Here, we generate a three-dimensional grid and place the previous simulated data within it, recording the number of points in each grid cell as the probability density. At the same time, we take the center point of each grid cell to represent the point in the discrete probability. This way, we obtain a discretized distribution $(S_{t_1},S_{t_2},v_{t_1})$ of the Heston model as the reference measure $Q$ (see Figure \ref{image:RefM-Q}). 

% {\color{red}
% Subsequently, to generate market data, we add some white noise to the base Heston model. The reason for adding noise is to simulate that real market data distributions cannot be directly fitted by the Heston model. The market data will demonstrate the distributions of stock prices and volatilities at various moments. We add different levels of white noise to each dimension of the data generated by the Heston model, creating three new arrays with market randomness. Considering the width of the grid, here we choose the appropriate coefficients $(\sigma_1,\sigma_2,\sigma_3)$. In fact, the noised marginal distributions can be written as
% \begin{align*}
% \mu &= \text{Law}\{ S_{t_1} + \sigma_1 \sqrt{t_1} \mathcal{N} (0,1)\} \\
% \nu &= \text{Law}\{ S_{t_2} + \sigma_2 \sqrt{t_2} \mathcal{N} (0,1)\}\\
% \rho &= \text{Law}\{ v_{t_1} + \sigma_3 \sqrt{t_1} \mathcal{N} (0,1)\}.
% \end{align*}
% And the Histograms are shown as Figure \ref{image:Histogram}.}

To simulate realistic market data, we perturb the outputs of the base Heston model by adding independent white noise. This addition accounts for the fact that real-world market distributions cannot be perfectly captured by the Heston dynamics alone. The resulting data better reflect the statistical uncertainties present in observed financial markets, particularly in the distributions of asset prices and volatilities at different time points.

Specifically, we introduce noise with different magnitudes in each dimension of the simulated data—corresponding to \( S_{t_1} \), \( S_{t_2} \), and \( v_{t_1} \)—to obtain three randomized samples incorporating market variability. The noise levels are controlled by coefficients \( (\sigma_1, \sigma_2, \sigma_3) \), chosen appropriately based on the grid resolution.
The resulting marginal distributions used in the calibration are then given by:
\begin{align*}
\mu &= \text{Law}\left( S_{t_1} + \sigma_1 \sqrt{t_1} \, \mathcal{N}(0,1) \right), \\
\nu &= \text{Law}\left( S_{t_2} + \sigma_2 \sqrt{t_2} \, \mathcal{N}(0,1) \right), \\
\rho &= \text{Law}\left( v_{t_1} + \sigma_3 \sqrt{t_1} \, \mathcal{N}(0,1) \right).
\end{align*}
The corresponding histograms of these noisy marginals are shown in Figure~\ref{image:Histogram}.

% \begin{figure}
% \FIGURE
% {\includegraphics{image/Histogram_of_stock_prices_at_time_t_1_and_t_2.png}}
% {Histogram of stock price at time $t_1$ and $t_2$ and volatility at time $t_1$. \label{image:Histogram}}{}
% \end{figure}

\begin{figure}[htbp]
\centering
\subfigure{
    \begin{minipage}{0.45\linewidth}
    \centering
    \includegraphics[width=\textwidth]{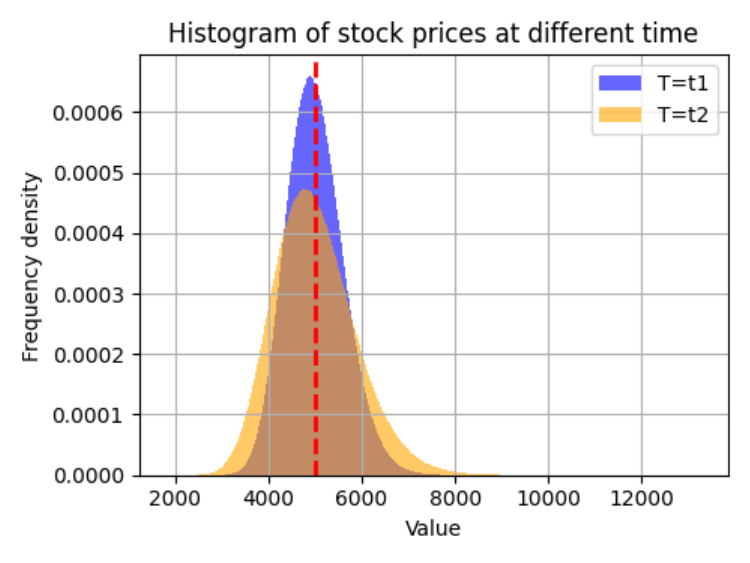}
    \end{minipage}
}
\subfigure{
    \begin{minipage}{0.45\linewidth}
    \centering
    \includegraphics[width=\textwidth]{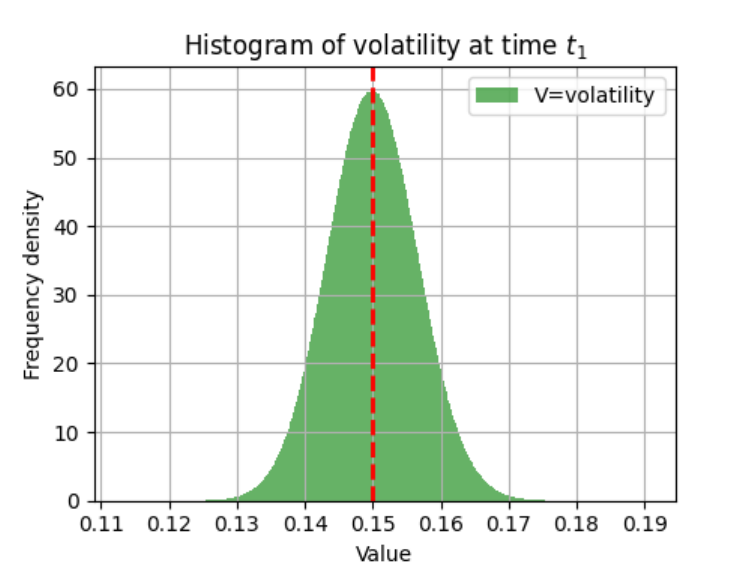}
    \end{minipage}
}
\centering
\caption{Histogram of stock prices at time $t_1$ and $t_2$ and volatility at time $t_1$}
\label{image:Histogram}
\end{figure}

Then, following the previous grid partitioning method, we extract the grid points corresponding to each dimension, transforming the data into three distributions $\mu$, $\nu$ and $\rho$. For Heston model generation, we chose $K = 8\times 10^7$ points for simulation, implemented in the Google Colab platform, with a running time of 2 minutes 51 seconds and using approximately 9.8 GB of RAM. To make it easier to organize the notation, we present the parameters required for generating market data and calculating the reference measure during the numerical experiments, which are summarized in Table \ref{Table:Describe heston}. 

\begin{table}
      \centering
      \begin{tabular}{c c c}
        \toprule
        \textbf{Parameters} & \textbf{Values/Dimensions} & \textbf{Explanation} \\
        \midrule
        $S_0$ & 5000 & Initial stock price at time 0 \\
        $v_0$ &  0.15 & Initial volatility at time 0\\
        $\bar{v}$ & 0.15 & Long-term mean of volatility\\
        $\lambda$ & 1 & Speed of reversion of $v_t$ to its long-term mean \\
        $\eta$ & 0.05 & Coefficient of vol of vol \\
        $t_1$ & 0.1 & First expiration date\\
        $t_2$ & 0.2 & Second expiration date\\
        $\Delta_{t}$ & 0.01 & Time step size\\
        $K$   & $8\times 10^7$ & Heston model simulations number\\
        $(\sigma_1,\sigma_2,\sigma_3)$ & (100,150,0.01)& Coefficient of the white noise \\
        % $\mu$    & \textcolor{red}{Dimension 40} & Market data distribution of stock price at time $t_1$ \\
        % $\nu$    & Dimension 50 & Market data distribution of stock price at time $t_2$ \\
        % $\rho$    & Dimension 5 & Market data distribution of volatility at time $t_1$ \\
        \bottomrule
      \end{tabular}
      \caption{Parameters and variables in Heston model}
      \label{Table:Describe heston}
\end{table}

\paragraph{Numerical Results}
Following the algorithm, we calculate the EMOT problem and obtain the optimized three-dimensional transition matrix. Here, we record all parameters of the algorithm in Table \ref{Table:Describe Sinkhorn}. The algorithm obtains the optimal transfer probability matrix $\pi^\star$. 
% \begin{table}
% \TABLE
% {Parameters and variables in Sinkhorn's algorithm\label{Table:Describe Sinkhorn}}
% {\begin{tabular}{@{}l@{\quad}c@{\quad}l@{}}
% \hline\up 
% \textbf{Parameters} & \textbf{Values} & \textbf{Explanation} \\ \hline\up 
%         vector $x$   &  [3437.5, 3512.5, 3587.5,...,6362.5] & From interval [3400,6400], pick $x_{dim}= N=40$ points \\
%         vector $y$   &  [3235, 3305, 3375,...,6665] & From interval [3200,6700], pick $y_{dim}=M=50$ points \\
%         vector $z$   &  [0.138, 0.144,...,0.162] & From interval [0.135,0.165], pick $z_{dim}=5$ points \\
%         $f_0$    & $0_{x_{dim}}$ & Initial point, dimension $x_{dim} = 40$ \\
%         $g_0$    & $0_{y_{dim}}$ & Initial point, dimension $y_{dim} = 50$ \\
%         $h_0$    & $0_{x_{dim}}$ & Initial point, dimension $x_{dim} = 40$ \\
%         T & 1000 & Total number of iterations \\
%         $c_{i,j,k}$ & $-\log(\frac{d Q}{d \mu \otimes \nu \otimes \rho})$  & Reference measure with dimensions$(40,50,5)$
% \down\\ \hline
% \end{tabular}}{}
% \end{table}

\begin{table}
      \centering
      \begin{tabular}{c c c }
        \toprule
        \textbf{Parameters} & \textbf{Values} & \textbf{Explanation} \\
        \midrule
        vector $x$   &  [3437.5, 3512.5, 3587.5,...,6362.5] & From interval [3400,6400], pick $x_{dim}= N=40$ points \\
        vector $y$   &  [3235, 3305, 3375,...,6665] & From interval [3200,6700], pick $y_{dim}=M=50$ points \\
        vector $z$   &  [0.138, 0.144,...,0.162] & From interval [0.135,0.165], pick $z_{dim}=5$ points \\
        $f_0$    & $0_{x_{dim}}$ & Initial point, dimension $x_{dim} = 40$ \\
        $g_0$    & $0_{y_{dim}}$ & Initial point, dimension $y_{dim} = 50$ \\
        $h_0$    & $0_{x_{dim}}$ & Initial point, dimension $x_{dim} = 40$ \\
        T & 1000 & Total number of iterations \\
        $c_{i,j,k}$ & $-\log(\frac{d Q}{d \mu \otimes \nu \otimes \rho})$  & Reference measure with dimensions$(40,50,5)$\\
        \bottomrule
      \end{tabular}
      \caption{Parameters and variables in Sinkhorn's algorithm}
      \label{Table:Describe Sinkhorn}
\end{table}

We need to verify whether the optimized transition probability matrix satisfies the marginal distribution and martingale conditions. From Figure \ref{image:MOT-marginal}, we can observe that the marginal distributions of matrix $\pi$ are consistent with $\mu$, $\nu$, and from Figure \ref{image:MOT-martingal} they satisfy the martingale condition within a magnitude of relative error at $10^{-6}$. Also, the optimal transport matrix $\pi$ with fixed volatility $Z=0.15$ can be shown as in Figure \ref{image:MOTmatrix}. 

From the perspective of the real market, for the given distributions $\mu$, $\nu$, and $\rho$, they do not satisfy the Heston model. Therefore, simply using SVM to calibrate the joint distribution will result in inaccurate marginal distributions. However, using the optimal joint distribution obtained from the EMOT problem not only has the same marginal distributions as the real market data but also satisfies the required martingale property for no-arbitrage. Comparing $\pi$ and $Q$ (see Figures \ref{image:RefM-Q} and \ref{image:MOTmatrix}), it can be seen a difference between the two distributions.

\begin{figure}[htbp]
\centering
\includegraphics[width=0.45\textwidth]{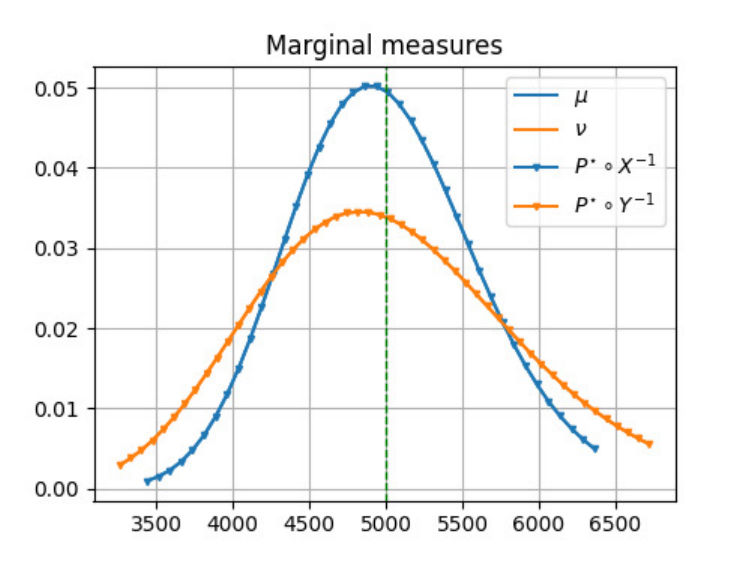}
\caption{Lines: density function of probability measures $\mu,\nu$; Dotted lines: marginal distribution $X,Y$ of optimal transport matrix $\pi^\star$.}
\label{image:MOT-marginal}
\end{figure}

\begin{figure}[htbp]
    \centering
    \begin{minipage}[h]{0.45\linewidth}
        \centering
        \includegraphics[width=\textwidth]{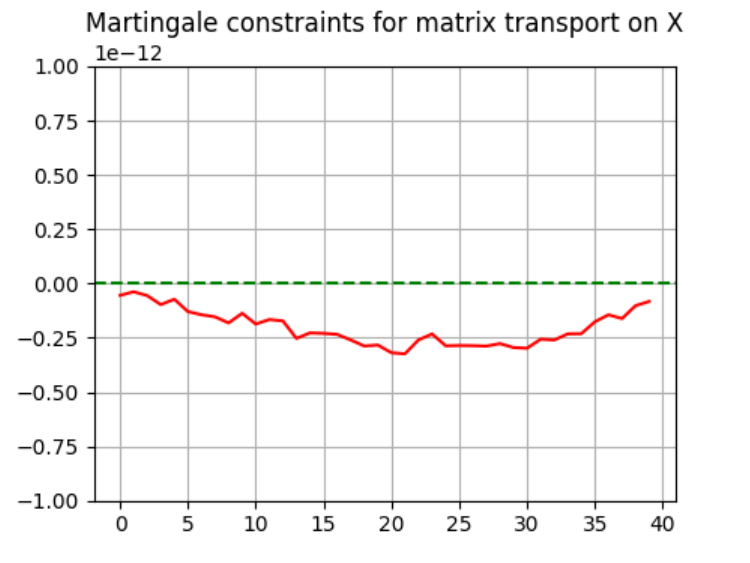}
        \caption{Martingale constraints: $\int_{\mathcal Y,\mathcal Z} (y-x)d\hat \pi_{y,z|x}$ for different $x$.}
        \label{image:MOT-martingal}
    \end{minipage}
    \hfill
    \begin{minipage}[h]{0.45\linewidth}
        \centering
        \includegraphics[width=\textwidth]{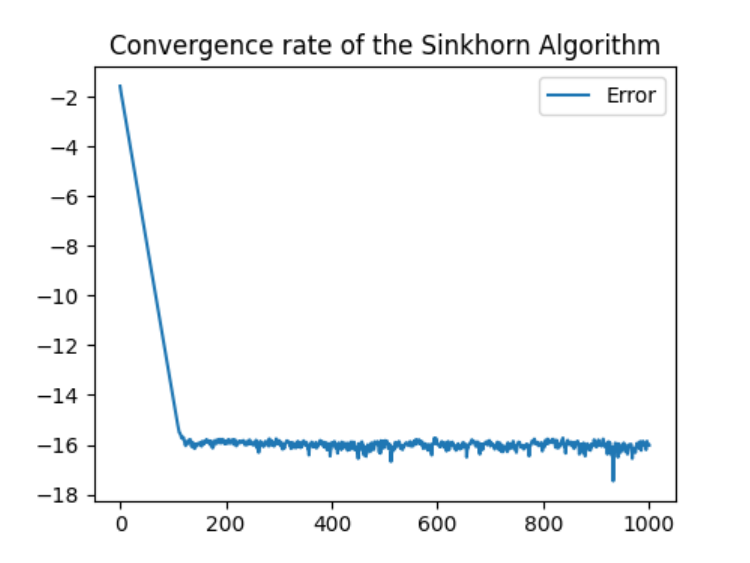}
        \caption{Convergence result: $\max_{s\in[T]} \mathcal{G}(f^s,g^s,h^s) - \mathcal{G} (f^t,g^t,h^t)$ versus the iteration step.}
        \label{image:Convergence}
    \end{minipage}
\centering
\end{figure}

\begin{figure}[htbp]
    \centering
    \begin{minipage}[h]{0.48\linewidth}
        \centering
        \includegraphics[width=\textwidth]{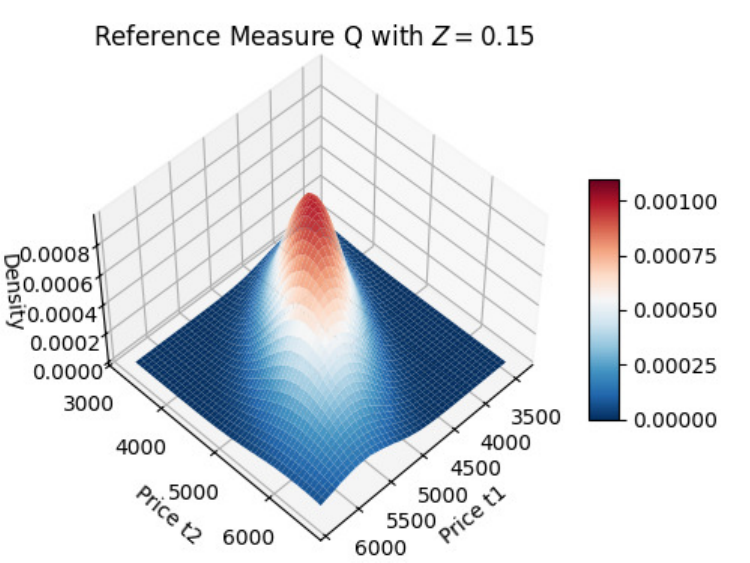}
        \caption{Reference measure $Q$ with fixed $Z=0.15$}
        \label{image:RefM-Q}
    \end{minipage}
    \hfill
    \begin{minipage}[h]{0.48\linewidth}
        \centering
        \includegraphics[width=\textwidth]{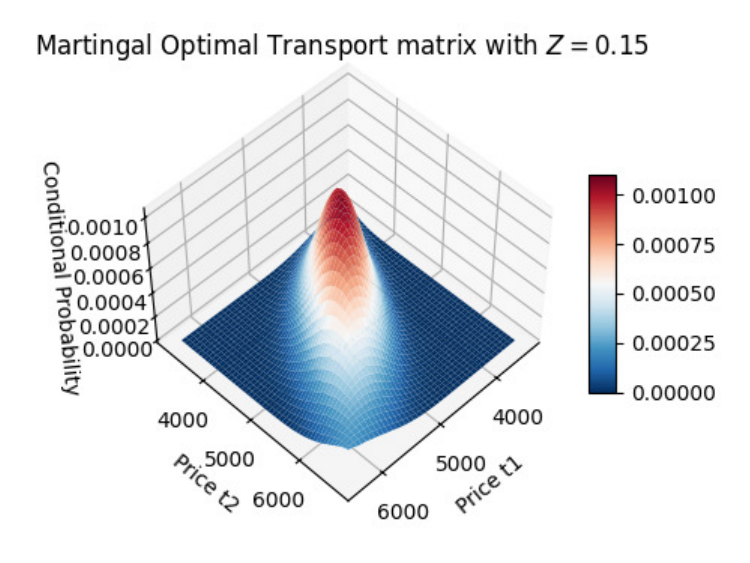}
        \caption{Martingale Optimal Transport probability matrix with fixed $Z=0.15$}
        \label{image:MOTmatrix}
    \end{minipage}
\centering
\end{figure}
It is noteworthy that we have recorded the latest Entropy results $H(\pi|Q)$ at each step during the iterative process of Sinkhorn's algorithm, and plotted a graph of its convergence speed over time (see Figure \ref{image:Convergence}). Clearly, during the 1000-step iteration process, the first 200 steps show stable convergence, and the convergence rate satisfies exponential convergence properties. And the subsequent process (200-1000 steps) is to make tiny oscillations around the optimal point, achieving the optimal solution. In the implementation process, we use the Google Colab platform to run the code, in which Sinkhorn's algorithm is computed in about 9 seconds (110 iterations/s), which has a high running efficiency.

\section{Proof of the Main Results}
\label{sec:main_result_proofs}

\subsection{Uniform bound for dual functions}
\label{proof:uniform_bound}
\begin{proof}{Proof of Proposition \ref{prop:existence_p1_conti}}
First, Assumption~\ref{assu:exsitence_mg_transport_conti} ensures that \( \nu \) is not the Dirac measure \( \delta_0 \) at 0 and satisfies the zero-mean condition \( \int_{\mathbb{R}} y \nu(dy) = 0 \). In particular, this implies that \( \nu((0, +\infty)) > 0 \) and \( \nu((-\infty, 0)) > 0 \). 
%Then we can find two constants $Y^u$ and $Y^d$ such that $\nu([Y^u,+\infty)) ,\nu((-\infty,Y^d]) > 0$ and $Y^u > 0 > Y^d$. 
Consequently, we can choose two constants \( Y^u > 0 \) and \( Y^d < 0 \) such that both \( \nu([Y^u, +\infty)) > 0 \) and \( \nu((-\infty, Y^d]) > 0 \). We construct $\hat \pi$ by perturbing $\bar \pi$ defined in Assumption \ref{assu:exsitence_mg_transport_conti}:
\begin{equation*}
    d\hat \pi(x,y,z) = \left\{
    \begin{array}{l}
        d\bar \pi(x,y,z) + \delta \mu((-\infty,0))\nu((-\infty,Y^d])  d \mu\otimes  \nu\otimes\rho , \text{\quad for\ \ } x\ge 0\  \text{and}\  y\ge Y^u; \\
        d\bar \pi(x,y,z) - \delta \mu([0,+\infty))\nu((-\infty,Y^d])   d \mu\otimes  \nu\otimes\rho, \text{\quad for\ \ } x< 0\  \text{and}\  y\ge Y^u; \\
       d\bar \pi(x,y,z) - \delta  \mu((-\infty,0))\nu([Y^u,+\infty)) d\mu\otimes \nu\otimes\rho , \text{\quad for\ \ } x\ge 0\  \text{and}\  y\le Y^d;\\
       d\bar \pi(x,y,z) +\delta \mu([0,+\infty))\nu([Y^u,+\infty))  d \mu\otimes  \nu\otimes \rho , \text{\quad for\ \ } x< 0\  \text{and}\  y\le Y^d;\\
       d\bar \pi(x,y,z), \quad \text{otherwise};
    \end{array}
    \right.
\end{equation*}
where $\delta$ is a positive constant defined by
$$\delta:= \frac{1}{2\big(\mu((-\infty,0))\vee\mu([0,+\infty))\big)\cdot  \big(\nu((-\infty,Y^d])\vee\nu([Y^u,+\infty))\big)} \cdot \exp(-|p|_{L^\infty(\mu\otimes \nu\otimes\rho)}).$$ 
It is easy to check from the construction and the choice of $\delta$ that 
\begin{equation}
\label{eq:pi1_is_positive}
    \frac{1}{2}\cdot \frac{d \bar \pi}{d \mu\otimes  \nu\otimes\rho}\le \frac{d\hat \pi}{d \mu\otimes  \nu\otimes\rho}\le  \frac{3}{2}\cdot \frac{d\bar \pi}{d \mu\otimes  \nu\otimes\rho}.
\end{equation}
Therefore, the quotient density $|\log d\hat \pi/d(\mu\otimes \nu\otimes\rho)|$ is bounded by Assumption \ref{assu:exsitence_mg_transport_conti}. We can verify the marginal distribution condition for $\hat \pi\in \Pi(\mu,\nu)$, using the equalities 
\begin{equation*}
\begin{aligned}
& \int_{\mathcal{Y} \times \mathcal{Z}} \left[ \hat\pi(dx, dy, dz) - \bar\pi(dx, dy, dz) \right] \cdot \mathbf{1}_{\{x \ge 0\}} \\
& =\delta\Big(  \int_{y\ge Y^u} \mu((-\infty,0))\nu((-\infty,Y^d])  d \nu -  \int_{y\le Y^d} \mu((-\infty,0))\nu([Y^u,+\infty)) d \nu\Big)\mu(dx) = 0,\\
\mbox{and}\quad & \int_{\mathcal{X} \times \mathcal{Z}} \left[ \hat\pi(dx, dy, dz) - \bar\pi(dx, dy, dz) \right] \cdot \mathbf{1}_{\{y > Y^u\}}  \\
& =\delta \Big(\int_{x\ge 0}  \mu((-\infty,0))\nu((-\infty,Y^d])  d\mu -   \int_{x< 0} \mu([0,+\infty))\nu((-\infty,Y^d])   d\mu \Big)\nu(dy)  = 0.
\end{aligned}
\end{equation*}

By a similar computation, the result can also be shown for the case $x<0$ or $y\le Y^d$, and therefore we show $\hat \pi \in \Pi(\mu,\nu)$. The purpose of this perturbation is to break the martingale constraint slightly, creating a uniform and strictly positive lower bound on the expectation of $(Y-X)$ conditional on $X$ (for $X\ge 0$, if $\text{Law}(X,Y)=\hat \pi$), while carefully preserving the prescribed marginal distributions. This strict positivity will be crucial for establishing uniform lower bounds in the dual convergence analysis.

The conditional increment  $\int_{\mathcal Y,\mathcal Z} (y-x)d\hat \pi_{y,z|x}$ for $x\ge 0$ can be calculated by noticing the difference
\begin{equation*}
\begin{aligned}
    \int_{\mathcal Y\times\mathcal Z} (y-x)d\hat \pi_{y,z|x} &- \int_{\mathcal Y\times \mathcal Z} (y-x)d\bar \pi_{y,z|x} \\
    = \delta  & \int_{y\ge Y^u} (y-x) \mu((-\infty,0))\nu((-\infty,Y^d])  \nu(dy) \\  
    & \quad \quad \quad \quad \quad - \delta \int_{\tilde y \le Y^d} (\tilde y-x)\mu((-\infty,0))\nu([Y^u,+\infty)) \nu(d\tilde y) \\
    = \delta & \int_{y\ge Y^u} y \mu((-\infty,0))\nu((-\infty,Y^d])  \nu(dy)  \\
    & \quad \quad \quad \quad \quad - \delta  \int_{\tilde y\le Y^d} \tilde y\mu((-\infty,0))\nu([Y^u,+\infty)) \nu(d\tilde y)\\
    \ge \delta & (Y^u-Y^d) \mu((-\infty,0)) \big(\nu([Y^u,+\infty)) \wedge \nu((-\infty,Y^d
    ])\big), \quad x>0,
\end{aligned}
\end{equation*}
% where we use the martingale property for $\bar \pi$ to yield the subtractor on the left cancels in the second equality. 
Similarly, for the case when $x<0$, the conditional increment for $\hat \pi$ can be computed and we summarize these two estimations to be
\begin{equation*}
\label{eq:pi_hat_epsilon_p}
    \int_{\mathcal Y\times \mathcal Z
    } (y-x)d\hat \pi_{y,z|x}    \left\{  \begin{array}{l}
         \ge\ \  \delta  (Y^u-Y^d) \mu((-\infty,0)) \big(\nu([Y^u,+\infty)) \wedge \nu((-\infty,Y^d
    ])\big) :=\epsilon_p, \text{\quad for\ } x \ge 0; \\
   \le -\delta  (Y^u-Y^d) \mu([0,+\infty)) \big(\nu([Y^u,+\infty)) \wedge \nu((-\infty,Y^d
    ])\big):=-\epsilon_p, \text{\quad for\ }  x < 0; \\
    \end{array}
    \right.
\end{equation*}
As a consequence, it is possible for us to find a constant $\epsilon_p$ to satisfy the required condition \eqref{eq:pi1_property_semi_mg}. 
% The martingale property for other dimensions preserves, to see this, we compute 
% \begin{equation*}
% \begin{aligned}
%     \int_{\mathcal Y} &(y_{i'}-x_{i'})d\hat \pi^{i}_{y|x} - \int_{\mathcal Y} (y_{i'}-x_{i'})d\bar \pi_{y|x} \\
%     & = \delta \cdot \int_{y:y_i\ge Y_i^u} y_{i'} \mu_i((-\infty,0))\nu_i((-\infty,Y_i^d])  \bar \nu(dy)  - \delta \cdot \int_{y:\tilde y_i <Y_i^d} \tilde y_{i'}\mu_i((-\infty,0))\nu_i([Y_i^u,+\infty)) \bar \nu(d\tilde y)\\ 
%     & = C\cdot \int y_{i'} \nu_i(dy_{i'}) = 0,
% \end{aligned}
% \end{equation*}
% for some constant factor $C$. 
Furthermore, the relative entropy is finite. From inequality \eqref{eq:pi1_is_positive}, we have $d\hat \pi/d\bar\pi\in[1/2,3/2]$, so that
\begin{equation*}
\begin{aligned}
     H(\hat \pi|Q) \le \int \Big|\log \frac{d\hat \pi}{dQ}\Big| d\bar \pi \le \int \Big|\log \frac{d\bar \pi}{dQ}\Big|  d\hat \pi + \log 2 \le \frac{3}{2}\int \Big|\log \frac{d\bar \pi}{dQ}\Big|  d\bar \pi+ \log 2.
\end{aligned}
\end{equation*}
For the right-hand side, we estimate
\begin{equation*}
\begin{aligned}
    \int\Big|\log \frac{d\bar \pi}{dQ}\Big|d\bar \pi &=  H(\bar \pi|Q) + 2\int\Big(\log \frac{d\bar \pi}{dQ}\Big)^-d\bar \pi \le H(\bar \pi|Q)  + 2e^{-1}. 
\end{aligned}
\end{equation*}
where the last inequality follows from the fact that $f\log f \ge -e^{-1}$. 
% \begin{remark}
%     It can be seen that the conditional increase will deteriorate or approach zero if the measure $\nu([\underline X,\overline X]^c)$ is small. It accords with our intuition, because when  $\nu([\underline X,\overline X]^c)\to 0$, it is hard for states close to the boundary $\overline X, \underline X$ to transport in a martingale manner. At that time we can imagine that the dual functions will surge near the boundary, and the convergence rate will slower
% \end{remark}
\end{proof}

In the following, we are going to first derive a uniform $L^1(\mu\otimes \nu\otimes\rho)$ bound for the dual coefficients $(f^n,g^n,h^n)_{n\ge 1}$ and then establish a uniform $L^\infty(\mu\otimes \nu\otimes\rho)$ bound. Those bounds are crucial for analyzing the convergence property of Sinkhorn's algorithm. Note that the estimates are stable under the invariant transforms \eqref{eq:invariant_transform}. 

\begin{proposition}\label{prop:l1_bound_conti}
     Under Assumptions \ref{assu:lipschitz_c_conti}, \ref{assu:exsitence_mg_transport_conti} and \ref{assu:measuable_init_bdd}, through Sinkhorn's iterations, the functions $(\tilde f^{n},\tilde g^{n},\tilde h^{n})_{n\ge 1}$ satisfy the uniform $L^1$-estimate: 
    \begin{equation*}
    \begin{aligned}
           &\sup_{n\ge1}|\tilde h^n-\tilde h^n(0)|_{L^1(\mu)}<+\infty,\\ 
           &\sup_{n\ge1}\int_{\mathcal X} \big|\tilde f^n(x)-\tilde f^n(0)-\tilde h^n(0) x \big| \mu(dx)<+\infty,\\
            &\sup_{n\ge1} \int_{\mathcal Y} \big|\tilde g^n(y) + \tilde f^n(0)+\tilde h^n(0) y \big| \nu(dy)<+\infty.
    \end{aligned}
    \end{equation*}
    In addition, we have the following uniform pointwise upper bound for the value $\tilde f^n(x)-\tilde f^n(0)-\tilde h^n(0) x$:
    \begin{equation}\label{eq:uniform-upperbound-f}
        \sup_{n\ge 1} \tilde f^n(x)-\tilde f^n(0)-\tilde h^n(0) x \le L_c|x|.
    \end{equation} 
\end{proposition}

\begin{proof}{Proof}
Fix any \( x, x' \in \mathcal{X}_0 \). We study the difference between the two dual potentials \(\tilde{f}^{n+1}(x)\) and \(\tilde{f}^{n+1}(x')\) at iteration \(n+1\), defined via the log-partition function:
\begin{equation*}
\begin{aligned}
    &\tilde f^{n+1}(x) - \tilde f^{n+1}(x') \notag \\
    &= \log \int_{\mathcal Y \times \mathcal Z} \exp\left(-c(x,y,z) - g^n(y) - \tilde h^{n+1}(x)(y - x)\right) \nu(dy)\rho(dz) \notag \\
    &\quad - \log \int_{\mathcal Y \times \mathcal Z} \exp\left(-c(x',y,z) - g^n(y) - \tilde h^{n+1}(x')(y - x')\right) \nu(dy)\rho(dz). 
\end{aligned}
\end{equation*}
Recall the $\pi(f, g, h)$ defined in \eqref{eq:definition_of_induced_prob}. Let us define a probability measure \( \pi_{y,z|x'} \in \mathcal{P}(\mathcal{Y} \times \mathcal{Z}) \)   equal to the conditional probability of $\pi(f, g^n, \tilde h^{n+1})$, that is,
\begin{equation*}
\pi_{y,z|x'}(dy,dz) := \frac{1}{Z(x')} \exp\left(-c(x',y,z) - g^n(y) - \tilde h^{n+1}(x')(y - x')\right) \nu(dy)\rho(dz),
\end{equation*}
where \( Z(x') = \int_{\mathcal Y \times \mathcal Z} \exp\left(-c(x',y,z) - g^n(y) - \tilde h^{n+1}(x')(y - x')\right) \nu(dy)\rho(dz) \) is the normalization constant making this a probability measure. Then, using Jensen's inequality, we obtain:
\begin{equation}\label{eq:jensen}
\begin{aligned}
\tilde f^{n+1}(x) - \tilde f^{n+1}(x') & = \log \int_{\mathcal Y \times \mathcal Z} \exp\left(-c(x,y,z) - g^n(y) - \tilde h^{n+1}(x)(y - x)\right) \nu(dy)\rho(dz) - \log Z(x')\\
& = \log \int_{\mathcal{Y}\times\mathcal{Z}}  
\exp\left(-\Delta_{x,x'}(y,z)\right) \pi_{y,z|x'}(dy,dz) \\
&  \ge - \int_{\mathcal{Y}\times\mathcal{Z}} \Delta_{x,x'}(y,z) \pi_{y,z|x'}(dy,dz),
\end{aligned}
\end{equation}
where we define
\[
\Delta_{x,x'}(y,z) := \big(c(x,y,z) - c(x',y,z)\big) + \big(\tilde h^{n+1}(x) - \tilde h^{n+1}(x')\big)(y - x).
\]
To estimate this quantity, we apply Assumption \ref{assu:lipschitz_c_conti} to have $|c(x,y,z) - c(x',y,z)| \le L_c |x - x'|$. Next, for the second term in \(\Delta_{x,x'}\), note that by the definitions of \(\pi_{y,z|x'}\) and $\tilde h^n$ in \eqref{eq:h_equality_conitnuous}, we have $\int_{\mathcal{Y}\times\mathcal{Z}} (y - x') \pi_{y,z|x'}(dy,dz)  = 0$. Hence, $\int_{\mathcal{Y}\times\mathcal{Z}} (y - x)\pi_{y,z|x'}(dy,dz) = \int_{\mathcal{Y}\times\mathcal{Z}} \big((y - x') + (x' - x)\big) \pi_{y,z|x'}(dy,dz) = x' - x$. Together with \eqref{eq:jensen} we have
\begin{equation*}
    \begin{aligned}
    \tilde f^{n+1}(x) - \tilde f^{n+1}(x') 
    &\ge - L_c |x - x'| - \left(\tilde h^{n+1}(x) - \tilde h^{n+1}(x')\right)(x' - x).
    \end{aligned}    
\end{equation*}
By exchanging the roles of $x$ and $x'$ and summing  up the two inequalities, we obtain
\begin{equation*}
    (x' - x)\left(\tilde h^{n+1}(x') - \tilde h^{n+1}(x)\right) \le 2 L_c |x - x'|.
\end{equation*}
That is,
\begin{equation}
\label{eq:h_bound_eq_conti}
    \left(\tilde h^{n+1}(x') - \tilde h^{n+1}(x)\right)  \mathrm{sgn}(x' - x) \le 2 L_c.
\end{equation}
Taking $x'=0$, we obtain $\tilde{h}^{n+1}(x) - \tilde{h}^{n+1}(0)\le 2L_c$ for $x\ge0$ and $\tilde{h}^{n+1}(x) - \tilde{h}^{n+1}(0)\ge -2L_c$ for $x<0$. Let $\hat \pi$ be the probability measure  constructed in Proposition \ref{prop:existence_p1_conti}. Note that
\begin{equation}
\label{eq:formula_H_p1_conti}
    \begin{aligned}
    H(\hat \pi|Q) + \int_{\mathcal X} \tilde{h}^{n+1}(x) \int_{\mathcal Y\times\mathcal Z}(y-x) d\hat \pi 
    & \ge \inf_{\pi\in \Pi(\mu,\nu)} \left\{H(\pi|Q) + \int_{\mathcal X} \tilde{h}^{n+1}(x)  \int_{\mathcal Y\times\mathcal Z} (y-x)d\pi \right\} \\
    & \ge \mathcal G(\tilde{f}^{n+1},g^n,\tilde{h}^{n+1}),
    \end{aligned}
\end{equation}
where the last inequality follows from the duality. In addition, $\mathcal G(\tilde{f}^{n+1},g^n,\tilde{h}^{n+1}) \ge \mathcal G(f^0,g^0,h^0)$ for all $n$. Further, by the compactness of support for $\mu$ and $\nu$,   there exists $C_0>0$ such that
\begin{equation}
\label{eq:basic_half_bound_ineq}
    \int_{\mathcal X} \Big(\tilde{h}^{n+1}(x) -\tilde{h}^{n+1}(0) -2L_c\textrm{sgn}(x)\Big)   \int_{\mathcal Y\times\mathcal Z} (y-x) \hat \pi_{y,z|x}(dy,dz) \hat \pi_x(dx) \ge -C_0.
\end{equation}
Recall that $\hat \pi_x = \mu$ and that $\hat \pi_{y,z|x}$ satisfies \eqref{eq:pi1_property_semi_mg}. Together with the argument after \eqref{eq:h_bound_eq_conti}, we obtain that the integrand on the left-hand side of \eqref{eq:basic_half_bound_ineq} is non-positive. Moreover,
\begin{equation*}
\begin{aligned}
    \epsilon_p \int_{\mathcal X} \Big|\tilde{h}^{n+1}(x)& - \tilde{h}^{n+1}(0) - 2L_c  \textrm{sgn}(x)\Big|\mu(dx) \\
    &\le \int_{\mathcal X} \big|\tilde{h}^{n+1}(x) -\tilde{h}^{n+1}(0) - 2L_c  \textrm{sgn}(x)\big|\Big| \int_{\mathcal Y\times\mathcal Z} (y-x)\hat \pi_{y,z|x}(dy,dz) \Big|\mu(dx) \le C_0,
\end{aligned}
\end{equation*}
so that it can be deduced that $|\tilde{h}^{n+1} - \tilde{h}^{n+1}(0) |$ is uniformly bounded in $L^1(\mu)$. 

Next, we consider the first-order equality of $\tilde{f}^{n+1}$, which writes,
\begin{equation*}
\log \int_{\mathcal Y\times\mathcal Z}\exp(-c(x,y,z)-g^n(y)-\tilde{h}^{n+1}(x) y)(\nu\otimes\rho)(dy,dz) = \tilde{f}^{n+1}(x)-x\tilde{h}^{n+1}(x).  
\end{equation*}
Together with the convexity inequality $\exp(x)-\exp(y)\ge \exp(y)(x-y)$, we estimate:% \begin{equation*}
% \begin{aligned}
% \label{eq:f_bound_eq_conti}
%         0 &=\int_{\mathcal Y}\Big(-\frac{dc(x,y)}{dx}-\frac{df^{n+1}(x)}{dx}-\frac{dh^{n+1}(x)}{dx}(y-x)+h^{n+1}(x)\Big)d\pi_{y|x}(dy) = 0\\
%         & = -\frac{dc(x,y)}{dx}-\frac{df^{n+1}(x)}{dx} + h^{n+1}(x)
% \end{aligned}
% \end{equation*}
\begin{equation}
\begin{aligned}
\label{eq:f_bound_eq_conti}
        & \quad \tilde{f}^{n+1}(x)-\tilde{f}^{n+1}(0)-\tilde{h}^{n+1}(0) x \\ 
        & = \tilde{f}^{n+1}(x)-\tilde{h}^{n+1}(x) x-\tilde{f}^{n+1}(0)+\big(\tilde{h}^{n+1}(x)-\tilde{h}^{n+1}(0)\big) x\\
        &\ge \int_{\mathcal Y\times\mathcal Z}\Big(c(0,y,z)-c(x,y,z)-\big(\tilde{h}^{n+1}(x)-\tilde{h}^{n+1}(0)\big) y\Big)\pi_{y,z|0}(dy)+\big(\tilde{h}^{n+1}(x)-\tilde{h}^{n+1}(0)\big) x \\
        & \ge -L_c|x| + \big(\tilde{h}^{n+1}(x)-\tilde{h}^{n+1}(0)\big) x.
\end{aligned}
\end{equation}
where the first inequality is due to the martingale property of $\pi_{y,z|x}(\cdot,g^n,\tilde{h}^{n+1})$. To derive the upper bound, we instead apply the reversed convexity estimate and observe
\begin{equation*}
\begin{aligned}
\label{eq:f_bound_eq_conti2}
& \quad \tilde{f}^{n+1}(x) -\tilde{f}^{n+1}(0)-\tilde{h}^{n+1}(0) x \\
&\le \int_{\mathcal Y\times\mathcal Z}\Big(c(0,y,z)-c(x,y,z)-\big(\tilde{h}^{n+1}(x)-\tilde{h}^{n+1}(0)\big) y\Big)\pi_{y,z|x}(dy)+\big(\tilde{h}^{n+1}(x)-\tilde{h}^{n+1}(0)\big) x \\
        & \le L_c|x| - \big(\tilde{h}^{n+1}(x)-\tilde{h}^{n+1}(0)\big) x + \big(\tilde{h}^{n+1}(x)-\tilde{h}^{n+1}(0)\big) x = L_c|x|.
\end{aligned}
\end{equation*}
Therefore, $\tilde{f}^{n+1}-\tilde{f}^{n+1}(0)-\tilde{h}^{n+1}(0) x$ is uniformly bounded in $L^1(\mu)$ norm, i.e.
\begin{equation}
\label{eq:f_tilde_bound}
    \begin{aligned}
        \big|\int_{\mathcal X} \tilde{f}^{n+1}-\tilde{f}^{n+1}(0)-\tilde{h}^{n+1}(0) x d\mu \big|<C_2
    \end{aligned}
\end{equation}
% , i.e., $|df^{n+1}(x)/dx|_{L^1(\mu)} \le C$, by the Assumption \ref{assu:lipschitz_c_conti} and the uniform $L^1(\mu)$ bound for $h^{n+1}$. 
% \textcolor{blue}{Further using the Assumption \ref{assu:measuable_init_bdd}, which provides a lower bound for the density $\mu(x)$ so that $f^{n+1}$ has uniformly bounded variation and thus is uniformly bounded due to the normalization. } 
Finally for $\tilde{g}^{n+1}$, again by the convexity of $z\mapsto \exp(z)$, we get:
\begin{equation*}
\begin{aligned}
        0 & = \int_{\mathcal X\times\mathcal Z} \Big( \exp \big(-c(x,y,z)-\tilde{f}^{n+1}(x)-\tilde{g}^{n+1}(y)\\
        &\quad \quad \quad \quad \quad \quad \quad \quad \quad \quad \quad \quad -\tilde{h}^{n+1}(x)(y-x) \big) -\exp \big(-p(x,y,z) \big) \Big) d(\mu\otimes\rho)\\
        &\ge \int_{\mathcal X\times\mathcal Z} \exp \big(-p(x,y,z)\big) \big( p(x,y,z)-c(x,y,z) \\
        &\quad \quad \quad \quad \quad \quad \quad \quad \quad \quad \quad \quad -\tilde{f}^{n+1}(x)-\tilde{g}^{n+1}(y)-\tilde{h}^{n+1}(x)(y-x) \big) d(\mu\otimes \rho)\\
        &\ge -|p|_{L^\infty} - |c|_{L^\infty} - \tilde{g}^{n+1}(y)-\tilde{f}^{n+1}(0)-\tilde{h}^{n+1}(0) y - \int_{\mathcal X}\big(\tilde{f}^{n+1}(x)-\tilde{f}^{n+1}(0) + \tilde{h}^{n+1}(0) x \big)d\mu\\
        &\quad \quad \quad \quad \quad \quad \quad \quad\quad \quad \quad \quad  -(|\mathcal X| + |\mathcal Y|)|\tilde{h}^{n+1}-\tilde{h}^{n+1}(0)|_{L^1(\mu)},
\end{aligned}
\end{equation*}
where $p =-\ln  \frac{d \bar\pi}{d\mu\otimes \nu\otimes\rho}$, and  $|\mathcal X|$ and $|\mathcal Y|$ are defined in Assumption \ref{assu:mu_nu}.
This implies that for any \( (x, y) \in \operatorname{supp}(\mu \otimes \nu) \).The inequality above implies a lower bound for the quantity: $\tilde{g}^{n+1}(y)+\tilde{f}^{n+1}(0)+\tilde{h}^{n+1}(0)\cdot y\ge -C_1$. Recalling that the function $\mathcal G(\tilde{f}^{n+1},\tilde{g}^{n+1},\tilde{h}^{n+1}) = -1 - \int_{\mathcal X} \tilde{f}^{n+1}d\mu - \int_{\mathcal Y} \tilde{g}^{n+1}d\nu$ and \ref{eq:f_tilde_bound}, we have
\begin{equation*}
\begin{aligned}
        0 & \leq \int_{\mathcal Y\times\mathcal Z} (\tilde{g}^{n+1}+\tilde{f}^{n+1}(0)+\tilde{h}^{n+1}(0) y + C_1)d(\nu\otimes\rho) \\
        & = C_1  - 1 + \big( 1 + \int_{\mathcal X} \tilde{f}^{n+1}d\mu + \int_{\mathcal Y} \tilde{g}^{n+1}d\nu \big) - \int_{\mathcal X} \big(\tilde{f}^{n+1}-\tilde{f}^{n+1}(0)-\tilde{h}^{n+1}(0) x \big)d\mu\\
        &= C_1 -1 - \mathcal G(\tilde{f}^{n+1},\tilde{g}^{n+1},\tilde{h}^{n+1}) -\int_{\mathcal X} \big(\tilde{f}^{n+1}-\tilde{f}^{n+1}(0)-\tilde{h}^{n+1}(0) x \big)d\mu.
\end{aligned}
\end{equation*}
It follows from the optimality condition for the dual function $\tilde{g}^{n+1}$ that $\mathcal G(\tilde{f}^{n+1},\tilde g^{n+1},\tilde{h}^{n+1}) \ge \mathcal G(\tilde{f}^{n+1},g^n,\tilde{h}^{n+1}) \ge \mathcal G(f^0,g^0,h^0)$. Therefore, $\tilde{g}^{n+1}+\tilde{f}^{n+1}(0)+\tilde{h}^{n+1}(0) y$ has a uniform bound in $L^1(\nu)$.
\end{proof}
\vspace{3mm}

\begin{proof}{Proof of Theorem \ref{thm:uniform_bound_conti}}
Given the uniform $L^1$-bound of the quantities $\big(\tilde h^n-\tilde h^n(0), \tilde f^n(x)-\tilde f^n(0)-\tilde h^n(0) x, \tilde g^n(y) + \tilde f^n(0)+\tilde h^n(0) y \big)_n$ established in Proposition \ref{prop:l1_bound_conti}, we now derive the uniform $L^\infty$-bound using the measure $\hat \pi$ constructed in Proposition \ref{prop:existence_p1_conti}. For $(\tilde{f}^{n+1},g^n,\tilde{h}^{n+1})$, the induced probability measure satisfies the martingale property with $\mathcal{X}$-marginal distribution $\pi_x = \mu$. Recall the probability constructed in  Proposition \ref{prop:existence_p1_conti}: $\hat \pi = \exp(-\hat p(x,y,z))d\mu\otimes \nu\otimes\rho\in \Pi(\mu,\nu)$. Since $\pi_x = \mu = \hat\pi_x$, we have
\begin{equation*}
\begin{aligned}
        0 & = \int_{\mathcal Y\times\mathcal Z} \Big(\exp \big(-c(x,y,z)-\tilde{f}^{n+1}(x)-g^{n}(y)-\tilde{h}^{n+1}(x)(y-x)\big) -\exp \big(-\hat p(x,y,z) \big) \Big) d(\nu\otimes\rho)\\
        &\ge \int_{\mathcal Y\times\mathcal Z} \exp \big(-\hat p(x,y,z)\big) \big(\hat p(x,y,z)-c(x,y,z) \\
        & \quad \quad \quad \quad \quad \quad \quad \quad \quad\quad \quad -\tilde{f}^{n+1}(x)-g^{n}(y)-\tilde{h}^{n+1}(x)(y-x)\big)d(\nu\otimes\rho)\\
        &\ge -|\hat p|_\infty - |c|_\infty -\int_{\mathcal Y \times\mathcal Z}\big( \tilde{f}^{n+1}(x)+ g^n(y) + \tilde{h}^{n+1}(x)(y-x) \big) d\hat \pi_{y,z|x}(x,dy,dz) \\
        & \ge -|\hat p|_\infty - |c|_\infty -\int_{\mathcal Y \times\mathcal Z}\big( \tilde{f}^{n+1}(x)+ g^n(y)+ \tilde{h}^{n+1}(0)(y-x)\big) d\hat\pi_{y,z|x}(x,dy,dz) \\
        & \quad \quad \quad \quad \quad \quad \quad \quad \quad\quad \quad -\int_{\mathcal Y\times Z}  \big(\tilde{h}^{n+1}(x) -  \tilde{h}^{n+1}(0)\big)  (y-x)d\hat \pi_{y,z|x}(x,dy,dz).
\end{aligned}
\end{equation*}
Applying the uniform  upper bound from \eqref{eq:uniform-upperbound-f}, we obtain
\begin{equation}
\label{eq:temp_g_semi_bound}
\begin{aligned}
        \quad \big(\tilde{h}^{n+1} & (x) - \tilde{h}^{n+1}(0)\big)\int_{\mathcal Y\times \mathcal Z}(y-x)d\hat \pi_{y,z|x}(x,dy,dz)\\
        & \ge -|\hat p|_\infty - |c|_\infty - |(\tilde{f}^{n+1}-\tilde{f}^{n+1}(0)-\tilde{h}^{n+1}(0) x)^+|_{L^{\infty}(\mu)} \\
        & \quad \quad \quad \quad \quad - |\tilde{g}^{n+1}+\tilde{f}^{n+1}(0)+\tilde{h}^{n+1}(0) y|_{L^1(\nu)}=: - C_2.
\end{aligned}
\end{equation}
Hence, for $x\ge0$, we have $\tilde{h}^{n+1}(x) - \tilde{h}^{n+1}(0)\ge -C_2/\epsilon_p$, where $\epsilon_p$ is given in Proposition \ref{prop:existence_p1_conti}. Similarly, for $x<0$, we have $\tilde{h}^{n+1}(x) - \tilde{h}^{n+1}(0)\le C_2/\epsilon_p$. Together with the bound in \eqref{eq:h_bound_eq_conti} we arrive at a uniform $L^\infty(\mu)$-bound for $\tilde{h}^{n+1}(x) - \tilde{h}^{n+1}(0)$. 

Next, we derive the uniform $L^\infty(\mu)$-bound for the quantities $\tilde{f}^{n+1}-\tilde{f}^{n+1}(0)-\tilde{h}^{n+1}(0) x$ and $\tilde{g}^{n+1}+\tilde{f}^{n+1}(0)+\tilde{h}^{n+1}(0)y$. For the former, we recall \eqref{eq:f_bound_eq_conti} to obtain its uniform $L^\infty(\mu)$-bound, and for the latter, we write
\begin{equation*}
\begin{aligned}
        \tilde{g}^{n+1}(y) &+ \tilde{f}^{n+1}(0)+\tilde{h}^{n+1}(0) y \\
        &= \log \int_{\mathcal X\times\mathcal Z} \exp\Big(-c(x,y,z)-\big(\tilde{f}^{n+1}(x)-\tilde{f}^{n+1}(0)-\tilde{h}^{n+1}(0) x\big) \\
        & \quad \quad \quad \quad \quad \quad \quad \quad \quad \quad \quad \quad -\big(\tilde{h}^{n+1}(x) - \tilde{h}^{n+1}(0)\big)(y-x)\Big)d(\mu\otimes\rho),
\end{aligned}
\end{equation*}

Finally, for the normalization of $(\tilde f^{n+1},\tilde g^{n+1},\tilde h^{n+1})\to (f^{n+1},g^{n+1},h^{n+1})$ by \eqref{eq:normalization_iter_n}, we compute
\begin{equation*}
    |h^{n+1}| =\Big|\tilde h^{n+1} - \int_{\mathcal X} \tilde h^{n+1}d\mu \Big|\le \Big|\tilde h^{n+1}- \tilde h^{n+1}(0) - \int_{\mathcal X} \tilde h^{n+1} - h^{n+1}(0)d\mu \Big|\le 2 |\tilde h^{n+1}-\tilde h^{n+1}(0)|_{L^\infty(\mu)}.
\end{equation*}
Using Assumption \ref{assu:mu_nu}, we also have
\begin{equation*}
\begin{aligned}
        |g^{n+1}| &= \Big|\tilde g^{n+1}-\int_{\mathcal Y}\tilde g^{n+1}d\nu+y\int_{\mathcal X} \tilde h^{n+1} d\mu \Big|\\
        & = \Big|\tilde g^{n+1}+\tilde f^{n+1}(0)+\tilde h^{n+1}(0)y-\int_{\mathcal Y}\big(\tilde g^{n+1}+\tilde f^{n+1}(0)+\tilde h^{n+1}(0)y\big)d\nu+y\int_{\mathcal X} \big( \tilde h^{n+1}-\tilde h^{n+1}(0) \big) d\mu \Big|\\
        & \le 2|\tilde g^{n+1}+\tilde f^{n+1}(0)+ \tilde h^{n+1}(0)y|_{L^\infty(\nu)} + |\mathcal Y||\tilde h^{n+1} -\tilde h^{n+1}(0)|_{L^\infty(\mu)} 
\end{aligned}
\end{equation*}
and 
\begin{equation*}
\begin{aligned}
        |f^{n+1}| &= \Big|\tilde f^{n+1}+\int_{\mathcal Y}\tilde g^{n+1}d\nu-x\int_{\mathcal X} \tilde h^{n+1} d\mu \Big|\\
        & = \Big|\tilde f^{n+1}-\tilde f^{n+1}(0)-\tilde h^{n+1}(0)x+\int_{\mathcal Y}\big(\tilde g^{n+1}+\tilde f^{n+1}(0) + \tilde h^{n+1}(0)y\big)d\nu-x\int_{\mathcal X} \big( \tilde h^{n+1}-\tilde h^{n+1}(0) \big) d\mu \Big|\\
        & \le |\tilde f^{n+1}-\tilde f^{n+1}(0)-\tilde h^{n+1}(0)x|_{L^\infty(\mu)}+|\tilde g^{n+1}+\tilde f^{n+1}(0)+ \tilde h^{n+1}(0)y|_{L^\infty(\nu)}  \\
        & \quad \quad+ |\mathcal Y||\tilde h^{n+1} -\tilde h^{n+1}(0)|_{L^\infty(\mu)} 
\end{aligned}
\end{equation*}
These uniform $L^\infty$-bounds on $\tilde{f}^{n+1}-\tilde{f}^{n+1}(0)-\tilde{h}^{n+1}(0) x$, $\tilde g^{n+1}+\tilde f^{n+1}(0) + \tilde h^{n+1}(0)y $, and $\tilde{h}^{n+1} - \tilde{h}^{n+1}(0)$ imply the desired uniform $L^\infty$-bound. 
\end{proof}
\begin{remark}
In our problem, the construction of $\hat{\pi}$ is crucial in our argument because it directly implies the bounds established in \eqref{eq:h_bound_eq_conti} and \eqref{eq:temp_g_semi_bound}. For the case of high dimensional EMOT problem when $\mathcal X,\mathcal Y=\mathbb R^d$, in order to mimic the argument of the Theorem \ref{thm:uniform_bound_conti}, we need but failed to construct a  $\hat \pi\sim \mu\otimes \nu\otimes\rho$ such that $H(\hat \pi|\mu\otimes \nu\otimes\rho)<+\infty$ and $\int_{\mathcal Y\times \mathcal Z}(y-x)\hat \pi_{y,z|x}(dy,dz) = \epsilon(x)x/|x|$ for $x\in\mathcal X$, $\mu$-a.s, where $\epsilon(x)$ has a uniform positive lower bound $\epsilon_0$, i.e., $\epsilon(x)\ge \epsilon_0$. 

% And then, it can be reached an inequality similar to 
% \eqref{eq:basic_half_bound_ineq}:
% \begin{equation*}
%     \int_{\mathcal X} \Big(\big(\tilde{h}^{n+1}(x) -\tilde{h}^{n+1}(0) \big)\cdot  \int_{\mathcal Y\times\mathcal Z} (y-x) \hat \pi_{y,z|x}(dy,dz)-2L_c\Big)  \hat \pi_x(dx) \ge -C_0,
% \end{equation*}
% We can use this inequality to get the $L^1(\mu)$-norm bound of $ \tilde{h}^{n+1}(x) -\tilde{h}^{n+1}(0) $ and further the almost surely uniform boundedness. However, constructing the probability measure $\hat \pi$ to satisfy such a condition differs essentially from the case $d=1$.
\end{remark}

\subsection{Exponential convergence}

The proof of Theorem \ref{thm:existence_minimizer_linear_convergence} begins by leveraging the uniform bounds for the dual potentials established in Theorem \ref{thm:uniform_bound_conti}, namely, $|f^n(x)|_{L^\infty(\mu)}\le C_f, |g^n(y)|_{L^\infty(\nu)}\le C_g$ and $|h^n(x)|_{L^\infty(\mu)}\le C_h$, so as to establish the exponential convergence of the sequence $\big(\mathcal{G}(f^n,g^n,h^n)\big)_n$. As a result, the sequence of functions $(f^n,g^n,h^n)_n$ converges to a triple $(f^*,g^*,h^*) \in (L^\infty(\mu) \times L^\infty(\nu) \times L^\infty(\mu))$. 
Subsequently, we verify the first-order optimality conditions for the primal problem \eqref{eq:mg_transport}, demonstrating that the limit triple $(f^*,g^*,h^*)$ induces a probability measure solving the original problem \eqref{eq:mg_transport}. Ultimately, this establishes the absence of any primal-dual gap within the duality framework.
% Subsequently, we verify the first-order optimaty condition for the primal problem \eqref{eq:mg_transport}, confirming that the triple $(f^*,g^*,h^*)$ induces a probability measure, which solves the original problem \eqref{eq:mg_transport}. Eventually it establishes the absence of a primal-dual gap in this duality correlation. 
\begin{lemma}
\label{lem:extended_bound} 
Under Assumptions \ref{assu:lipschitz_c_conti}, \ref{assu:exsitence_mg_transport_conti} and \ref{assu:measuable_init_bdd}, for any $n\in \mathbb N^+$ we define the set $\mathcal I^n$:
\begin{equation*}
    \mathcal I^n:=\{(f^n,g^n,h^n),(f^n,g^n,\tilde{h}^{n+1}),(\tilde{f}^{n+1},g^n,\tilde{h}^{n+1}),(\tilde{f}^{n+1},\tilde{g}^{n+1},\tilde{h}^{n+1})\}.
\end{equation*}
There exist positive constants $C_f'\ge C_f,C_g'\ge C_g,C_h' \ge C_h$ such that for any $n\in \mathbb{N}^+$ and  any $(f,g,h)\in \mathcal I^n$
\begin{equation*}
\begin{aligned}
        |h&|_{L^\infty(\mu)} \le C_h', \quad  |f|_{L^\infty(\mu)} \le C_f', \quad |g|_{L^\infty( \nu)} \le C_g'.
\end{aligned}
\end{equation*}
% In particular, for all $n\in \mathbb{N}^+$, by the definition \eqref{eq:normalization_iter_n} the triple $(f^n, g^n, h^n)$ satisfies
% \[|h^n|_{L^\infty(\mu)} \le C_h', \quad  |f^n|_{L^\infty(\mu)} \le C_f' , \quad |g^n|_{L^\infty( \nu)} \le C_g'.\]
\end{lemma}
\begin{proof}{Proof}
The uniform bound for the first triple in $\mathcal I^n$ has already been given in Theorem \ref{thm:uniform_bound_conti}. 
Regarding the second triple $(f^n, g^n, \tilde h^{n+1})$, Proposition \ref{prop:well_defined_fgh} combined with the uniform $L^\infty(\nu)$-boundedness of $(g^{n})_n$ guarantees the existence of a constant $c_h>0$ independent of $n$, such that $|\tilde h^{n+1}(0)|\le c_h$. Using this bound and the result from Theorem \ref{thm:uniform_bound_conti}, we directly obtain
\begin{equation*}
    \begin{aligned}
        |\tilde{h}^{n+1}|_{L^\infty(\mu)}\le |\tilde{h}^{n+1}&-\tilde{h}^{n+1}(0)|_{L^\infty(\mu)}+|\tilde h^{n+1}(0)| \le C_h + c_h.
    \end{aligned}
\end{equation*}
Next, we consider the triple $(\tilde{f}^{n+1},g^n,\tilde{h}^{n+1})$. By the definition of $\tilde f^{n+1}(0)$ given in \eqref{eq:update_formula_fg}, we have the direct estimate that $|\tilde f^{n+1}(0)|\le |c|_{L^{\infty}}+C_g + c_h |\mathcal Y|$. Combining this estimate with the bound from Theorem \ref{thm:uniform_bound_conti}, we deduce
\begin{equation*}
\begin{aligned}
|\tilde f^{n+1} |_{L^\infty(\mu)} &\le |\tilde{f}^{n+1}-\tilde{f}^{n+1}(0)-\tilde{h}^{n+1}(0) x|_{L^\infty(\mu)} +|\tilde f^{n+1}(0)|+ |\mathcal X||\tilde h^{n+1}(0)| \\
&\le C_f + |c|_{L^{\infty}}+C_g + c_h |\mathcal Y| + |\mathcal X|c_h.
\end{aligned}
\end{equation*}
And also
\begin{equation*}
\begin{aligned}
    |\tilde g^{n+1}(y)|_{L^\infty(\nu)} &\le |\tilde{g}^{n+1}+\tilde{f}^{n+1}(0)+\tilde{h}^{n+1}(0) y|_{L^\infty(\nu)} +|\tilde f^{n+1}(0)|+ |\mathcal Y||\tilde h^{n+1}(0)|\\
    &\le C_g + |c|_{L^{\infty}}+C_g + c_h |\mathcal Y| + |\mathcal Y|c_h.    
\end{aligned}
\end{equation*}
Thus, setting $C_f':=C_f + |c|_{L^{\infty}}+C_g + c_h |\mathcal Y| + |\mathcal X|c_h, C_g':=C_g + |c|_{L^{\infty}}+C_g + c_h |\mathcal Y| + |\mathcal Y|c_h$ and $C_h':=C_h+c_h$, we obtain the desired uniform bounds for all the triples in $\mathcal I^n$. 
\end{proof}
The proof of exponential convergence shares some similarities with \cite{carlier2022linear}, which studies Sinkhorn's algorithm for entropic optimal transport without the martingale constraint. However, the martingale constraint and the corresponding dual term $h(x)(y-x)$ bring the major difficulties.

Define $C_0:=C_f'+C_g'+C_h'(|\mathcal X|+|\mathcal Y|)+|c|_{L^\infty}$. Then we have
\begin{equation*}
    |c(x,y,z)+f(x)+g(y)+h(x)(y-x)|\le C_0.
\end{equation*}
Besides, denoting $\kappa := \exp(-C_0)$, the density of probability measure $\pi(f,g,h)$ with respect to base measure $\mu\otimes \nu\otimes\rho$ can be bounded from both sides, that is, 
\begin{equation}
\label{eq:density_exp_estimate}
    \kappa \le\frac{d\pi(f,g,h)}{d(\mu\otimes \nu\otimes\rho)} = \exp(-c(x,y,z)-f(x)-g(y)-h(x)(y-x)) \le \kappa ^{-1},
\end{equation}
 for arbitrary $(f,g,h)\in \mathcal I^n$. 
 We define the set $\mathcal S$ of functions:
 \begin{equation*}
 \begin{aligned}
 \mathcal{S}:=\{& (f, g, h):(f, g, h) \in(L^{\infty}(\mu) \times L^{\infty}(\nu) \times L^{\infty}(\mu)),|f|_{L^\infty(\mu)} \leq C_f', \\
& \quad\quad\quad\quad|g|_{L^\infty(\nu)} \leq C_g',|h|_{L^\infty(\mu)} \leq C_h'\}.
\end{aligned}     
\end{equation*}
Lemma \ref{lem:extended_bound} yields that for any $n\in \mathbb N^+$, the set $\mathcal I^n\subset \mathcal S$.

For a function $\mathcal F: (L^\infty(\mu)\times  L^\infty(\nu) \times L^\infty(\mu))\to \mathbb R$, let us define the linear functional derivatives $\frac{\delta \mathcal F}{\delta f}, \frac{\delta \mathcal F}{\delta g}, \frac{\delta \mathcal F}{\delta h}: {( L^\infty(\mu)\times L^\infty(\nu) \times L^\infty(\mu)) } \times \mathbb{R} \to \mathbb{R}$, such that for any $(f,g,h),(f^{\prime},g^{\prime},h^{\prime}) \in (L^\infty(\mu)\times L^\infty(\nu) \times L^\infty(\mu)) $, it holds:
\begin{equation*}
    \begin{aligned}
        \mathcal F(f^{\prime},g^{\prime},h^{\prime}) - \mathcal F(f,g,h & ) =\int_0^1 \Big[\int_{\mathbb R} \frac{\delta \mathcal F}{\delta f}\big(f^t,g^t,h^t\big)(x)  (f^{\prime}- f)(x) \mu(dx) \\ 
        &+ \int_{\mathbb R} \frac{\delta \mathcal F}{\delta g}\big(f^t,g^t,h^t \big)(y)  (g^{\prime}- g)(y) \nu(dy) \\
        & + \int_{\mathbb R} \frac{\delta \mathcal F}{\delta h}\big(f^t,g^t,h^t \big)(x)  (h^{\prime}- h)(x) \mu(dx)\Big] d t.
    \end{aligned}
\end{equation*}
for $f^t = tf' + (1-t)f, g^t = tg' + (1-t)g, h^t = th' + (1-t)h$ and $0\le t\le 1$.

\begin{lemma}
\label{lem:linear_func_upper}
    Under Assumption \ref{assu:lipschitz_c_conti}, for any $(f,g,h)$, $(\tilde f,\tilde g,\tilde h)\in\mathcal S$, the functional derivatives of $\mathcal G(f,g,h)$ satisfy the following inequalities:
    \begin{equation}
    \label{eq:linear_func_upper}
    \begin{aligned}
            &\Bigg|\frac{\delta \mathcal G}{\delta f}(\tilde f,\tilde g,\tilde h)(x) - \frac{\delta \mathcal G}{\delta f}( f,g,h)(x)\Bigg|_{L^2(\mu)}^2\le M_f^2 \big|\tilde f-f+ \tilde g-g+(\tilde h-h)(y-x)\big|_{L^2(\mu\otimes\nu)}^2, \quad \\
            &\Bigg|\frac{\delta \mathcal G}{\delta g}( \tilde f,\tilde g,\tilde h)(y) - \frac{\delta \mathcal G}{\delta g}( f,g,h)(y)\Bigg|_{L^2(\nu)}^2\le M_g^2 \big|\tilde f-f+ \tilde g-g+(\tilde h-h)(y-x)\big|_{L^2(\mu\otimes\nu)}^2,\\
            &\Bigg|\frac{\delta \mathcal G}{\delta h}(\tilde  f,\tilde g,\tilde h)(x) - \frac{\delta \mathcal G}{\delta h}( f,g,h)(x)\Bigg|_{L^2(\mu)}^2\le M_h^2  \big|\tilde f-f+ \tilde g-g+(\tilde h-h)(y-x)\big|_{L^2(\mu\otimes\nu)}^2,
    \end{aligned}
    \end{equation}
    where $M_f = M_g = \kappa^{-1}$, and $M_h = \kappa^{-1}(|\mathcal X| + |\mathcal Y|)$.
    % \begin{equation}
    % \label{eq:concave_ineq_G}
    % \begin{aligned}
    %      \mathcal G(\tilde f,\tilde g,\tilde h) - \mathcal G( f,g,h) \le &\int_{\mathcal X}\frac{\delta \mathcal G}{\delta h}(f,g,h)\big(\tilde h(x)-h(x)\big)\mu(dx) - \frac{m_h}{2} |h-\tilde h|_{L^2(\mu)}^2\\   
    %      &+ \int_{\mathcal X}\frac{\delta \mathcal G}{\delta f}( f, g,\tilde h)\big(\tilde f(x)-f(x)\big)\mu(dx) - \frac{m_f}{2} |f-\tilde f|_{L^2(\mu)}^2\\
    %      &+ \int_{\mathcal Y}\frac{\delta \mathcal G}{\delta g}(\tilde f, g,\tilde h)\big(\tilde g(x)-g(x)\big)\nu(dy) - \frac{m_g}{2} |g-\tilde g|_{L^2(\nu)}^2,
    % \end{aligned}        
    % \end{equation}
\end{lemma}

\begin{proof}{Proof}
% Denote $\Delta(x,y) = f(x) + g(y) + h(x)(y-x)$ and $d\pi(x,y,z)= \exp(-c(x,y,z) - \Delta(x,y))d(\mu\otimes\nu\otimes\rho)$ is the positive measure induced by dual functional triple $(f,g,h)$, 
Let us consider the linear functional derivative of $\mathcal G$ with respect to $(f,g,h)$ under measure $\mu$ and $\nu$:
\begin{equation}
\label{eq:nabla_fgh}
\begin{aligned}
    \frac{\delta \mathcal G}{\delta f}(f,& g,h)(x) = \int_{\mathcal Y\times\mathcal Z}{\exp\big(-c(x,y,z)- f(x) - g(y) - h(x)(y-x) \big)(\nu\otimes\rho)(dy,dz)}-1,\\
    \frac{\delta \mathcal G}{\delta g}(f,& g,h)(y)  = \int_{\mathcal X\times\mathcal Z}{\exp\big(-c(x,y,z)- f(x) - g(y) - h(x)(y-x) \big)(\mu\otimes\rho)(dx,dz)}-1,\\
    \frac{\delta \mathcal G}{\delta h}(f,& g,h)(x) = \int_{\mathcal Y\times\mathcal Z}{(y-x)\exp\big(-c(x,y,z)- f(x) - g(y) - h(x)(y-x)\big)(\nu\otimes\rho)(dy,dz)},
\end{aligned}
\end{equation} 
defined for $ x\  \mu\text{-}a . s ., y\  \nu\text{-}a . s .$. For inequalities \eqref{eq:linear_func_upper}, we prove the third one as an example and the rest can be shown similarly. Using the compactness of the support for $\mu$ and $\nu$, and the elementary inequality
\begin{equation*}
    |\exp(x)-\exp(y)| \le \exp(x\vee y) |x-y|, \text{ for all } x,y\in\mathbb R
\end{equation*}
We immediately have
\begin{equation*}
\begin{aligned}
    &\quad \big|\exp(-c(x,y,z)-\tilde f(x)-\tilde g(y)-\tilde h(x)(y-x)) \\
    & \quad \quad \quad \quad \quad \quad \quad \quad \quad \quad-\exp(-c(x,y,z)-f(x)-g(y)- h(x)(y-x))\big|\\
    & \le \kappa^{-1} |\tilde f(x)-f(x)+ \tilde g(y)-g(y)+(\tilde h-h)(x)(y-x)|, \quad\text{for\ }x\  \mu\text{-}a.s..
\end{aligned}
\end{equation*}
It follows that
\begin{equation*}
\begin{aligned}
    \int_{\mathcal Y\times\mathcal Z}(y-x)&\Big(\exp\big(-c(x,y,z)- \tilde f(x) - \tilde g(y) - \tilde h(x)(y-x)\big)\\
    &\quad \quad \quad - \exp\big(-c(x,y,z)- f(x) - g(y) - h(x)(y-x)\big)\Big)
    (\nu\otimes\rho)(dy,dz)\\
    & \le (|\mathcal X| + |\mathcal Y|)\kappa ^{-1} \int_{\mathcal Y} \big|\tilde f(x)-f(x)+ \tilde g(y)-g(y)+(\tilde h-h)(x)(y-x)\big|\nu(dy).
\end{aligned}
\end{equation*}
Together with \eqref{eq:nabla_fgh}, we conclude the third inequality in \eqref{eq:linear_func_upper}. The rest of the inequalities can be shown similarly. 
\end{proof}

\begin{lemma}
\label{lem:strong_convave_G}
    For any $f,h\in L^\infty(\mu)$, $g\in L^\infty(\nu)$ the following equality for the function $\mathcal G(f,g,h)$ holds:
    \begin{equation}
    \label{eq:concave_equality}
    \begin{aligned}
        \mathcal G&(\tilde f,\tilde g,\tilde h) - \mathcal G(f,g,h) \\
        & - \int_{\mathcal X\times \mathcal Y} \big[ \frac{\delta \mathcal G}{\delta f}(f,g,h)(x)(\tilde f(x)-f(x)) + \frac{\delta \mathcal G}{\delta g}(f,g,h)(y)(\tilde g(y)-g(y))  \\
        & \quad \quad\quad\quad\quad\quad\quad\quad\quad\quad\quad\quad\quad\quad\quad + \frac{\delta \mathcal G}{\delta h}(f,g,h)(x)(\tilde h(x)-h(x)) \big] d(\mu\otimes\nu)\\
        &= -\int_0^1\int_{\mathcal X\times \mathcal Y\times \mathcal Z} \big|\tilde f(x)-f(x)+\tilde g(y)-g(y)+(\tilde h(x)-h(x))(y-x)\big|^2 d\pi(f^t,g^t,h^t) dt.
    \end{aligned}
    \end{equation}
    In particular, $\mathcal G(f,g,h)$ is concave for $f,h\in L^\infty(\mu)$, $g\in L^\infty(\nu)$.
   \end{lemma}
   % {\color{red} In particular, $\mathcal G(f,g,h)$ is strictly concave up to invariant transforms \eqref{eq:invariant_transform}, i.e.,
   % \begin{equation}
   % \label{eq:concave_def}
   %     \mathcal G(\alpha f+(1-\alpha)\tilde f,\alpha g+(1-\alpha) \tilde g,\alpha h + (1-\alpha) \tilde h)\ge \alpha \mathcal G(f,g,h) + (1-\alpha) \mathcal G(\tilde f,\tilde g,\tilde h),
   % \end{equation}
   % where the equality is only attained when $(\tilde f, \tilde g, \tilde h) = (f + c_1 - c_2x, g - c_1 + c_2y, h - c_2)$ for some constants $c_1, c_2 \in \mathbb{R}$.}
\begin{proof}{Proof}
    By the definition of the linear functional derivative,
    \begin{equation*}
    \begin{aligned}
        \mathcal G(\tilde f,\tilde g,\tilde h) - &\mathcal G(f,g,h) = \int_0^1\int_{\mathcal X\times \mathcal Y} \Big[\frac{\delta \mathcal G}{\delta f }(x)(f^t, g^t,h^t)(\tilde f(x)-f(x)) \\
        &+ \frac{\delta \mathcal G}{\delta g }(f^t,g^t,h^t)(y) (\tilde g(y)-g(y))
        + \frac{\delta \mathcal G}{\delta h }(f^t,g^t,h^t)(x)(\tilde h(x)-h(x)) \Big] d(\mu\otimes\nu) dt,
    \end{aligned}
    \end{equation*}
    where $(f^t,g^t,h^t) = (tf+(1-t)\tilde f, tg+(1-t)\tilde g, th+(1-t)\tilde h)$. Define the function $\mathcal H^1:[0,1]\mapsto \mathbb R$ as
    \begin{equation*}
    \begin{aligned}
        \mathcal H^1(t) = \int_{\mathcal X\times \mathcal Y} \Big[\frac{\delta \mathcal G}{\delta f }(& f^t,g^t,h^t)(x)(\tilde f(x)-f(x)) + \frac{\delta \mathcal G}{\delta g }(f^t,g^t,h^t)(y) (\tilde g(y)-g(y)) \\
        &+ \frac{\delta \mathcal G}{\delta h }(f^t,g^t,h^t)(x)(\tilde h(x)-h(x)) \Big] d(\mu\otimes\nu).
    \end{aligned}
    \end{equation*}
    By \eqref{eq:nabla_fgh} we can compute the derivative of $\mathcal H^1$ as
    \begin{equation*}
        \frac{d\mathcal H^1(t)}{dt} = -\int_{\mathcal X\times \mathcal Y\times \mathcal Z} \big|\tilde f(x)-f(x)+\tilde g(y)-g(y)+(\tilde h(x)-h(x))(y-x)\big|^2 d\pi(f^t,g^t,h^t),
    \end{equation*}
    where the measure is defined as $\pi(f^t,g^t,h^t): = \exp(-c(x,y,z)-f^t(x)-g^t(y)-h^t(x)(y-x))$. By combining the equalities above, we obtain \eqref{eq:concave_equality}. Since the right-hand side of \eqref{eq:concave_equality} is clearly non-positive, it follows immediately that the functional $\mathcal G(f,g,h)$ is concave. 
\end{proof}

Let us define the set of normalized dual functions $(f,g,h)$:
\begin{equation*}
    \mathcal S^N = \big\{(f,g,h)\in \mathcal S: \int_{\mathcal Y}g(y)d\nu = \int_{\mathcal X}h(x)d\mu = 0\big\}
\end{equation*}
\begin{lemma} 
\label{lem:strong_convave_G2}Let Assumption \ref{assu:lipschitz_c_conti} hold true. 
Consider any pair of functions $(f,g,h),(\tilde f,\tilde g,\tilde h)\in\mathcal S^N$ satisfying the normalization conditions. Then there exist constants $m_f,m_g,m_h >0$ such that
\begin{equation}
\label{eq:complete_square_estimation}
\begin{aligned}
        \int_{\mathcal X\times \mathcal Y\times \mathcal Z} \big|\tilde f(x)&-f(x)+\tilde g(y)-g(y)+(\tilde h(x)-h(x))(y-x)\big|^2 d\pi(f^t,g^t,h^t) \\
        &\ge m_f\int_{\mathcal X}|\tilde f(x)-f(x)|^2d\mu + m_g\int_{\mathcal Y}|\tilde g(y)-g(y)|^2d\nu + m_h\int_{\mathcal X}|\tilde h(x)-h(x)|^2d\mu,
\end{aligned}
\end{equation}
 for any $(f^t,g^t,h^t):=(t f+(1-t)\tilde f,t g+(1-t) \tilde g,t h + (1-t) \tilde h)\in\mathcal S$ and $0\le t\le 1$. From inequality \eqref{eq:concave_equality}, it immediately follows that $\mathcal G(f,g,h)$ is strictly concave over the set $\mathcal S^N$. Together with the inequality \eqref{eq:concave_equality}, we obtain the following estimate on the increment of $\mathcal{G}$:
\label{lem:Gconcavity}
     \begin{equation}
    \label{eq:concave_ineq_G}
    \begin{aligned}
         \mathcal G(\tilde{f},&\tilde{g},\tilde{h}) - \mathcal G( f,g,h) \\
         \le & \frac{1}{4m_h}  \left|\frac{\delta \mathcal G}{\delta h} (f, g,h) \right|_{L^2(\mu)}^2 + \frac{1}{4m_f}  \left|\frac{\delta \mathcal G}{\delta f} ( f, g,h) \right|_{L^2(\mu)}^2 +\frac{1}{4m_g}  \left|\frac{\delta \mathcal G}{\delta g} ( f,  g,h) \right|_{L^2(\nu)}^2.
    \end{aligned}        
    \end{equation}
\end{lemma}
\begin{proof}{Proof}
To begin with, we notice the lower bound of quotient $d\pi(f^t,g^t,h^t)/d(\mu\otimes\nu\otimes \rho)$, and denote $ f'=\tilde f-f,  g' = \tilde g-g,  h' = \tilde h - h$, we deduce 
\begin{equation}
\label{eq:complete_square_estimation_temp1}
\begin{aligned}
    \int_{\mathcal X\times \mathcal Y\times \mathcal Z} \big| f'(x)&+ g'(y)+ h'(x)(y-x)\big|^2 d\pi(f^t,g^t,h^t) \\
    &\ge \kappa  \int_{\mathcal X\times \mathcal Y} \big|f'(x)+ g'(y)+ h'(x)(y-x)\big|^2 d\mu\otimes \nu  \\
    &= \kappa \Big(\int_{\mathcal X} \big|f'(x)-h'(x)x\big|^2 d\mu +  \int_{\mathcal X\times \mathcal Y} \big|g'(y)+h'(x)y\big|^2 d\mu\otimes\nu \\
    & =  \kappa \Big(\int_{\mathcal X} \big|f'(x)-h'(x)x\big|^2 d\mu +  \int_{\mathcal Y} \big|g'(y)\big|^2 d\nu + \text{Var}(\nu)\int_{\mathcal X} \big|h'(x)|^2d\mu\Big),
\end{aligned}
\end{equation}
where  we used the facts $\int_{\mathcal Y} g'(y)d\nu = \int_{\mathcal X} h'(x)d\mu = 0$, $\int_{\mathcal Y}yd\nu =0$,  and hence
\begin{equation*}
    \begin{aligned}
         \int_{\mathcal X\times \mathcal Y} (f'(x)-h'(x)x) \big(g'(y)+h'(x)y\big) d\mu\otimes \nu = 0 \quad \text{and} \quad \int_{\mathcal X\times \mathcal Y} g'(y)h'(x) yd\mu\otimes \nu = 0.
    \end{aligned}
\end{equation*}
Using the inequality $(a-b)^2 \ge (1-\delta) a^2 + (1-\frac{1}{\delta}) b^2$ valid for any $0<\delta<1$, 
\begin{equation*}
\begin{aligned}
    \int_{\mathcal X}|f'(x)-h'(x)x|^2 d\mu &\ge \int_{\mathcal X}(1-\delta)|f'(x)|^2 d\mu  - \big(\frac{1}{\delta}-1\big)\int_{\mathcal X}|h'(x)x|^2 d\mu \\
    &\ge \int_{\mathcal X}(1-\delta)|f'(x)|^2 d\mu  - \big(\frac{1}{\delta}-1\big)|\mathcal X|^2\int_{\mathcal X}|h'(x)|^2 d\mu.  
\end{aligned}
\end{equation*}
It follows from \eqref{eq:complete_square_estimation_temp1} that
\begin{equation*}
\begin{aligned}
     \int_{\mathcal X\times \mathcal Y\times \mathcal Z}& \big| f'(x)+ g'(y)+ h'(x)(y-x)\big|^2 d\pi(f^t,g^t,h^t) \\
     &\ge \kappa \Big((1-\delta)\int_{\mathcal X}|f'(x)|^2d\mu + \int_{\mathcal Y}|g'(y)|^2d\nu + \Big(-\big(\frac{1}{\delta}-1\big)|\mathcal X|^2+\text{Var}(\nu)\Big)\int_{\mathcal X}|h'(x)|^2d\mu\Big).
\end{aligned}
\end{equation*}
We take $\delta = (\text{Var}(\nu)/2|\mathcal X|^2 + 1)^{-1}$ and with the above estimation we let $m_f = \kappa (1-(\text{Var}(\nu)/2|\mathcal X|^2 + 1)^{-1})$, $m_g = \kappa $ and $m_h = \kappa \text{Var}(\nu)/2$ to satisfy \eqref{eq:complete_square_estimation}.
For inequality \eqref{eq:concave_ineq_G}, we notice that
    \begin{equation}
    \begin{aligned}
        \mathcal G(\tilde f,\tilde g,\tilde h) &- \mathcal G(f,g,h) - \int_{\mathcal X\times \mathcal Y} \big[ \frac{\delta \mathcal G}{\delta f}(f,g,h)(x)(\tilde f(x)-f(x)) + \frac{\delta \mathcal G}{\delta g}(f,g,h)(y)(\tilde g(y)-g(y)) \\
        & \quad \quad \quad \quad + \frac{\delta \mathcal G}{\delta h}(f,g,h)(x)(\tilde h(x)-h(x)) \big] d(\mu\otimes\nu)\\
        &\le -m_f\int_{\mathcal X}|\tilde f(x)-f(x)|^2d\mu - m_g\int_{\mathcal Y}|\tilde g(y)-g(y)|^2d\nu - m_h\int_{\mathcal X}|\tilde h(x)-h(x)|^2d\mu,
    \end{aligned}
    \end{equation}
    and the desired estimation \eqref{eq:concave_ineq_G} follows from Young's inequality.   
\end{proof}
Next, we need to establish uniform bounds on the derivatives of $(\tilde{f}^n, \tilde{g}^n , \tilde h ^n )_{n\in\mathbb N^+}$, which will subsequently allow us to extract a convergent subsequence using the Arzelà–Ascoli theorem.

\begin{lemma}
\label{lem:uniform_bdd_deri}
    Under Assumptions \ref{assu:lipschitz_c_conti}, \ref{assu:exsitence_mg_transport_conti} and \ref{assu:measuable_init_bdd}, the derivatives for $(\tilde f^n,\tilde g^n,\tilde h^n)_{n\in\mathbb N^+}$ are uniformly bounded, i.e., 
    \begin{equation*}
        \sup_{n\in \mathbb N^+} \Big|\frac{d\tilde f^{n}(x)}{dx}\Big|_{L^\infty(\mu)} < +\infty,\   \sup_{n\in \mathbb N^+} \Big|\frac{d\tilde g^n(y)}{dy}\Big|_{L^\infty(\nu)} < +\infty, \ {\text{and}}\   \sup_{n\in \mathbb N^+} \Big|\frac{d\tilde h^{n}(x)}{dx}\Big|_{L^\infty(\mu)} < +\infty.
    \end{equation*}
\end{lemma}
\begin{proof}{Proof}
    First, we compute the derivatives of $\tilde h^{n+1}$. By the first order condition \eqref{eq:h_equality_conitnuous} we have  
\begin{equation}
\label{eq:temp_mg_prop}
\begin{aligned}
    0 = \int_{\mathcal Y\times\mathcal Z} (y-x)\exp(-c(x,y,z)-{ g}^{n}(y)-\tilde h^{n+1}(x)y)(\nu\otimes\rho)(dy,dz).
\end{aligned}
\end{equation}
 By differentiating the equality \eqref{eq:temp_mg_prop} with respect to $x$, we obtain
\begin{equation*}
     \begin{aligned}
    0  = \int_{\mathcal Y\times\mathcal Z} &(y-x)\exp \big(-c(x,y,z)-{ g}^{n}(y)-\tilde h^{n+1}(x)y\big)\Big(-\frac{\partial c}{\partial x}(x,y,z) -y\frac{d\tilde h^{n+1}(x)}{dx}\Big)\\
    & (\nu\otimes\rho)(dy,dz) - \int_{\mathcal Y\times\mathcal Z} \exp(-c(x,y,z)-{ g}^{n}(y)-\tilde h^{n+1}(x)y)(\nu\otimes\rho)(dy,dz).
\end{aligned}
\end{equation*}
For fixed $x\in \mathcal X_0$, define:
\begin{equation*}
    \frac{d\pi}{d(\nu\otimes \rho)}(y,z):=\frac{\exp(-c(x,y,z)-g^n(y)-\tilde h^{n+1}(x)y)}{\int_{\mathcal Y\times \mathcal Z} \exp(-c(x,y,z)-g^n(y)-\tilde h^{n+1}(x)y)d(\nu\otimes \rho)}.
\end{equation*}
It follows
\begin{equation*}
    1 = \int_{\mathcal Y \times \mathcal Z} (y-x)\Big(-\frac{\partial c}{\partial x}(x,y,z) -y\frac{d\tilde h^{n+1}(x)}{dx}\Big)\pi(dy,dz).
\end{equation*}
Using the martingale property \eqref{eq:temp_mg_prop} of $\pi$, we obtain
\begin{equation}
\label{eq:temp_diff_h}
      \frac{d\tilde h^{n+1}(x)}{dx}\cdot \int_{\mathcal Y \times \mathcal Z} (y-x)^2\pi(dy,dz) = -1 - \int_{\mathcal Y \times \mathcal Z} (y-x)\frac{\partial c}{\partial x}(x,y,z)\pi(dy,dz).
\end{equation}
On the other hand, due to Proposition \ref{prop:well_defined_fgh} and Theorem \ref{thm:uniform_bound_conti}, the functions $\tilde h^{n+1}(x)$ and $g^n(y)$ are uniformly bounded. Hence, we obtain the following lower bound:
\begin{equation*}
\begin{aligned}
    \int_{\mathcal Y \times \mathcal Z} (y-x)^2\pi(dy,dz)& \ge \int_{\mathcal Y} (y-x)^2d\nu \exp(-|c|_{L^\infty}-C_g-C_h|\mathcal Y) \\
    &\ge {\text Var}(\nu)\exp(-|c|_{L^\infty}-C_g-C_h|\mathcal Y).
\end{aligned}
\end{equation*}
Together with \eqref{eq:temp_diff_h} and the fact that the right hand side of \eqref{eq:temp_diff_h}  is bounded by $1+L_c(|\mathcal X| + |\mathcal Y|)$, the derivative $|d\tilde h^{n+1}/dx|$ has a uniform bound. Further, the derivative of $d\tilde f^{n+1}/dx$:
\begin{equation*}
    \frac{d\tilde f^{n+1}}{dx}(x)=\int_{\mathcal Y\times \mathcal Z}\Big(-\frac{\partial c}{\partial x}(x,y,z) -\frac{d\tilde h^{n+1}}{dx}(x)(y-x)+\tilde  h^{n+1}(x)\Big )d\pi,
\end{equation*}
can be bounded uniformly  by $|c|_{L^\infty}+|d\tilde h^{n+1}/dx|_{L^\infty(\mu)}(|\mathcal X| + |\mathcal Y|) + |\tilde h^{n+1}|_{L^\infty(\mu)} $. In a similar manner, one can also show that the derivative $d\tilde g^{n+1}/dy$ admits a uniform bound. 

\end{proof}
\begin{proof}{Proof of Theorem \ref{thm:existence_minimizer_linear_convergence}} We divide the proof to several steps.

\vspace{0.1cm}

\noindent \textit{1. Exponential convergence of $\mathcal G$.}
We estimate the increment of $\mathcal G(f^n,g^n,h^n)$ by Sinkhorn's algorithm at each step:
\begin{equation}
\label{eq:descend_rate_h}
\begin{aligned}
    \mathcal G(f^{n+1},&g^{n+1},h^{n+1}) - 
    \mathcal G(f^{n},g^{n},h^{n}) = \big (\mathcal G({f}^{n},g^{n},\tilde{h}^{n+1}) - 
    \mathcal G(f^{n},g^{n},h^{n})\big) \\
    & +\big (\mathcal G(\tilde{f}^{n+1},g^{n},\tilde{h}^{n+1}) - 
    \mathcal G(f^{n},g^{n},\tilde h^{n+1})\big) + \big(\mathcal G(\tilde{f}^{n+1},\tilde{g}^{n+1},\tilde{h}^{n+1}) - 
    \mathcal G(\tilde{f}^{n+1},g^{n},\tilde{h}^{n+1})\big).
\end{aligned}
\end{equation}
Estimating the value of $\mathcal G$ around $(\tilde{f}^{n+1},\tilde{g}^{n+1},\tilde{h}^{n+1})$ as in the first inequality of \eqref{eq:concave_equality}, we obtain
\begin{equation}
\label{eq:descend_rate_h2}
\begin{aligned}
    \mathcal G&(\tilde{f}^{n+1},g^{n},\tilde{h}^{n+1}) - \mathcal G(\tilde{f}^{n+1},\tilde{g}^{n+1},\tilde{h}^{n+1}) \\
    &\le-\int_{\mathcal Y}\frac{\delta \mathcal G}{\delta g}(\tilde{f}^{n+1}, \tilde{g}^{n+1},\tilde{h}^{n+1}) (\tilde{g}^{n+1} - g^{n}) \nu(dy)  - \kappa\left|\tilde{g}^{n+1} - g^n\right|_{L^2(\nu)}^2 \\
    & = -\kappa\left|\tilde{g}^{n+1} - g^n\right|_{L^2(\nu)}^2,
\end{aligned}
\end{equation}
Because of the optimality of $\tilde{g}^{n+1}$, the functional derivative $\frac{\delta \mathcal G}{\delta g} (\tilde{f}^{n+1},\tilde{ g}^{n+1},\tilde{h}^{n+1})$ equals to zero. Similarly, the inequality related to the update of $f^n$ can be obtained as follows,
\begin{equation}
\label{eq:descend_rate_h3}
\mathcal G(f^{n},g^{n},\tilde h^{n+1}) - \mathcal G(\tilde{f}^{n+1},g^{n},\tilde{h}^{n+1}) \le -\kappa|\tilde{f}^{n+1}-f^n|_{L^2(\mu)}^2.
\end{equation}
While for the update of $\tilde h^{n+1}$: 
\begin{equation}
\label{eq:descend_rate_h4}
\begin{aligned}
    \mathcal G({f}^{n},g^{n},\tilde{h}^{n+1}) &- \mathcal G(f^{n},g^{n},h^{n})\le -\kappa\int_{\mathcal X\times \mathcal Y}|(\tilde h^{n+1}-h^n)(y-x)|^2d\mu \otimes \nu\\
    & = - \kappa\int_{\mathcal X}|(\tilde h^{n+1}-h^n)x|^2d\mu  - \kappa\text{Var}(\nu)\int_{\mathcal X}|\tilde h^{n+1}-h^n|^2d\mu \\
    & \le -\kappa\text{Var}(\nu)\int_{\mathcal X}|\tilde h^{n+1}-h^n|^2d\mu.
\end{aligned}
\end{equation}
These three inequalities characterize the ascent behavior at each iteration of Sinkhorn's algorithm. On the other hand, by the coercivity property stated in Lemma \ref{lem:Gconcavity}, we have that for any $(\tilde f,\tilde g,\tilde h)\in \mathcal S^N$,
\begin{equation*}
\begin{aligned}
    & \quad \mathcal G(\tilde f,\tilde g,\tilde h)- \mathcal G(f^{n},g^{n},h^{n})\\
& \le \frac{1}{4m_h}  \left|\frac{\delta \mathcal G}{\delta h} (f^{n}, g^{n},h^{n}) \right|_{L^2(\mu)}^2 + \frac{1}{4m_f}  \left|\frac{\delta \mathcal G}{\delta f} ( f^n, g^{n},h^{n}) \right|_{L^2(\mu)}^2 +\frac{1}{4m_g}  \left|\frac{\delta \mathcal G}{\delta g} ( f^n,  g^n,h^{n}) \right|_{L^2(\nu)}^2\\
    & \le\frac{1}{4m_h}\left|\frac{\delta \mathcal G}{\delta h} (f^{n}, g^{n},h^{n}) - \frac{\delta \mathcal G}{\delta h} (f^{n}, g^{n},\tilde{h}^{n+1})\right|_{L^2(\mu)}^2 \\
    & \quad + \frac{1}{4m_f} \left|\frac{\delta \mathcal G}{\delta f} (f^{n}, g^{n},h^{n}) - \frac{\delta \mathcal G}{\delta f} (\tilde f^{n+1}, g^{n},\tilde{h}^{n+1})\right|_{L^2(\mu)}^2\\
    & \quad +  \frac{1}{4m_g}\left|\frac{\delta \mathcal G}{\delta g} (f^{n}, g^{n},h^{n}) - \frac{\delta \mathcal G}{\delta g} (\tilde{f}^{n+1}, \tilde{g}^{n+1},\tilde{h}^{n+1})\right|_{L^2(\nu)}^2.
\end{aligned}
\end{equation*}
Taking supremum over all $(\tilde f,\tilde g,\tilde h)\in \mathcal S$ with normalization conditions on both sides, we get
\begin{equation*}
\begin{aligned}
    \mathcal G^{**} -\mathcal G(f^{n},& g^{n},h^{n}) := \sup_{(f,g,h)\in\mathcal S^N}
\mathcal{G}(f, g, h) - \mathcal{G}(f^n, g^n, h^n) \\
    &\le \frac{1}{4m_h}\left|\frac{\delta \mathcal G}{\delta h} (f^{n}, g^{n},h^{n}) - \frac{\delta \mathcal G}{\delta h} (f^{n}, g^{n},\tilde{h}^{n+1})\right|_{L^2(\mu)}^2  \\
    & \quad + \frac{1}{4m_f} \left|\frac{\delta \mathcal G}{\delta f} (f^{n}, g^{n},h^{n}) - \frac{\delta \mathcal G}{\delta f} (\tilde f^{n+1}, g^{n},\tilde{h}^{n+1})\right|_{L^2(\mu)}^2\\
    & \quad +  \frac{1}{4m_g}\left|\frac{\delta \mathcal G}{\delta g} (f^{n}, g^{n},h^{n}) - \frac{\delta \mathcal G}{\delta g} (\tilde{f}^{n+1}, \tilde{g}^{n+1},\tilde{h}^{n+1})\right|_{L^2(\nu)}^2.
\end{aligned}
\end{equation*}
 By inequality \eqref{eq:linear_func_upper} of Lemma \ref{lem:linear_func_upper} we have
\begin{equation*}
    \begin{aligned}
        \left|\frac{\delta \mathcal G}{\delta h} (f^{n}, g^{n},h^{n}) - \frac{\delta \mathcal G}{\delta h} (f^{n}, g^{n},\tilde{h}^{n+1})\right|_{L^2(\mu)}^2\le M_h^2(|\mathcal X| + |\mathcal Y|)^2|h^n-\tilde{h}^{n+1}|^2_{L^2(\mu)}.
    \end{aligned}
\end{equation*}
Bounds for the other two terms can be derived similarly. Combining the inequalities above, we obtain 
\begin{equation}
\label{eq:coercivity}
    \begin{aligned}
     \mathcal G^{**} - \mathcal G(f^{n},g^{n},h^{n}) \le & \Big(\frac{M_f^2}{4m_f}+\frac{M_h^2}{4m_h}\Big)|f^n-\tilde{f}^{n+1}|_{L^2(\mu)}^2 + \frac{M_g^2}{4m_g}|g^n-\tilde{g}^{n+1}|_{L^2(\nu)}^2 \\
     & + (|\mathcal X| + |\mathcal Y|)^2\Big(\frac{M_f^2}{4m_f} + \frac{M_g^2}{4m_g}+\frac{M_h^2}{4m_h}\Big)|h^n-\tilde{h}^{n+1}|_{L^2(\mu)}^2.
    \end{aligned}
\end{equation}

Combining the ascent estimates \eqref{eq:descend_rate_h2} and \eqref{eq:descend_rate_h3} with the coercivity inequality \eqref{eq:coercivity}, we obtain the exponential convergence. Define the constant 
\begin{equation}
\label{eq:C_s}
C_s = \max\left\{\Big(\frac{M_f^2}{4m_f}+\frac{M_h^2}{4m_h}\Big)\kappa^{-1}, \frac{M_g^2}{4m_g}\kappa^{-1}, \Big(\frac{M_f^2}{4m_f} + \frac{M_g^2}{4m_g}+\frac{M_h^2}{4m_h}\Big)\big(\kappa\text{Var}(\nu)\big)^{-1}\right\}
\end{equation}
and we have
\begin{equation}
\begin{aligned}
    \mathcal G^{**}-& \mathcal{G}(f^{n+1},g^{n+1},h^{n+1}) \\
    & \le C_s\big(\mathcal{G}(f^{n+1},g^{n+1},h^{n+1}) - \mathcal{G}(f^{n},g^{n},h^{n})\big) \\
    & =  C_s\Big(\mathcal G^{**} - \mathcal{G}(f^{n},g^{n},h^{n}) - \big(\mathcal G^{**}- \mathcal{G}(f^{n+1},g^{n+1},h^{n+1})\big)\Big),
\end{aligned}
\end{equation}
which implies that 
\begin{equation}
\label{eq:linear_cvgz}
    \mathcal G^{**} - \mathcal{G}(f^{n+1},g^{n+1},h^{n+1}) \le \Big(1 - \frac{1}{C_s+1}\Big) \Big(\mathcal G^{**} - \mathcal{G}(f^{n},g^{n},h^{n}) \Big).
\end{equation}
This will imply the desired exponential convergence of Sinkhorn's algorithm, once we prove $\mathcal G^* = \mathcal G^{**}$.
\vspace{0.1cm}

\noindent\textit{2. Existence of limit dual $(f^*,g^*,h^*)$.}
%However, it does not directly imply the $L^2$ convergence of the triple $(f^n,g^n,h^n)_n$, because we did a normalization from $(\tilde f^{n+1},\tilde g^{n+1},\tilde h^{n+1}) \to (f^{n+1},g^{n+1},h^{n+1})$.
To prove there exists a convergent subsequence of $f^n,h^n$ on domain $\mathcal X_0$ and $g^n$ on $ \mathcal Y_0$, we are going to use the Arzelà–Ascoli theorem. By the uniform $L^\infty$ boundedness of the triples $(\tilde f^{n+1},\tilde g^{n+1},\tilde h^{n+1})_n$ and their derivatives, as we proved in Lemma \ref{lem:uniform_bdd_deri}, there exists a subsequence of $(\tilde f^{n_k}\times\tilde g^{n_k}\times\tilde h^{n_k})_k$ uniformly converges to $(\tilde f^*,\tilde g^*,\tilde h^*)\in(L^\infty(\mu), L^\infty(\nu), L^\infty(\mu))$. 

Meanwhile, the normalized coefficients $(f^{n_k},g^{n_k},h^{n_k})_k$ defined in  \eqref{eq:normalization_iter_n} also converge and we denote their limit as $(f^*,g^*,h^*)=(\tilde f^*+\int_{\mathcal Y}\tilde g^*d\nu -x \int_{\mathcal X}\tilde h^*d\mu,\tilde g^*-\int_{\mathcal Y}\tilde g^*d\nu +y \int_{\mathcal X}\tilde h^*d\mu,\tilde h^*-\int_{\mathcal X}\tilde h^*d\mu )$. In particular, $( f^*, g^*, h^*)\in \mathcal S^N\subset(L^\infty(\mu)\times L^\infty(\nu)\times L^\infty(\mu))$.
\vspace{0.1cm}

\noindent\textit{3. Optimality of $(f^*,g^*,h^*)$ and $\mathcal G^*=\mathcal G^{**}.$}
First we show that 
$$ \mathcal G(f^*,g^*,h^*) = \mathcal G^{**} = \sup_{(f,g,h)\in\mathcal S^N}\mathcal{G}(f,g,h). $$ 
Recall the definition of $\mathcal G$,
\begin{equation*}
\begin{aligned}
       \mathcal G(f^n,g^n,h^n) = & \int_{\mathcal X\times\mathcal Y\times \mathcal Z} \exp\big(-c(x,y,z)-f^n(x)-g^n(y)-h^n(x)(y-x)\big)d\mu\otimes \nu\otimes\rho \\ 
       &- \int_{\mathcal X} f^n d\mu - \int_{\mathcal Y} g^n d\nu.
\end{aligned}
\end{equation*}
Recall that $(f^n,g^n,h^n)$ are uniformly bounded. Applying the bounded convergence theorem to the equality above, we have $\mathcal G(f^*,g^*,h^*)=\lim_k \mathcal G(f^{n_k},g^{n_k},h^{n_k})=G^{**}$. In order to prove $\mathcal G^*=\mathcal G^{**}$, that is, the triple $(f^*,g^*,h^*)$ is a solution of the dual problem without constraint $(f,g,h)\in \mathcal S$, we are going to verify the first-order conditions for the dual problem:
\begin{equation}
\label{eq:temp_h_FOC}
    \int_{\mathcal Y\times\mathcal Z} (y-x) \exp(-c(x,y,z)-f^{*}(x)-g^*(y)-h^{*}(x) (y-x))d\nu\otimes\rho = 0,
\end{equation}
 for $x\  \mu\text{-}a.s.$, as well as the first-order conditions for variables $f,g$:
 \begin{equation}
\label{eq:temp_fg_FOC}
\begin{aligned}
    &\int_{\mathcal Y\times\mathcal Z} \exp(-c(x,y,z)-f^{*}(x)-g^*(y)-h^{*}(x) (y-x))d\nu\otimes\rho = 1, \ \ \text{for}\ x\ \mu\text{-}a.s.\\
    &\int_{\mathcal X\times\mathcal Z} \exp(-c(x,y,z)-f^{*}(x)-g^*(y)-h^{*}(x) (y-x))d\mu\otimes\rho = 1, \ \ \text{for}\ y\ \nu\text{-}a.s..
\end{aligned}
\end{equation}
In fact, in Sinkhorn's iteration the dual coefficients satisfy 
\begin{equation*}
    \int_{\mathcal Y\times\mathcal Z} (y-x) \exp(-c(x,y,z)-f^{n+1}(x)-g^n(y)-h^{n+1}(x) (y-x))d\nu\otimes\rho = 0.
\end{equation*}
Then \eqref{eq:temp_h_FOC} follows from the bounded convergence theorem and the uniform boundedness of coefficients. The first-order conditions for $f,g$ variables \eqref{eq:temp_fg_FOC} can be shown similarly. Recall from Lemma \ref{lem:strong_convave_G} that the mapping $(f,g,h)\mapsto \mathcal G(f,g,h)$ is strongly concave on $\mathcal S^N$. Hence, $(f^*,g^*,h^*)$ is the unique solution to the dual problem \eqref{eq:dual_problem} under the normalization conditions. By this uniqueness, we conclude that the sequence of potential triples \((f^n, g^n, h^n)\) converges and can only converge to \((f^*, g^*, h^*)\).
Thus, we have proved the first two arguments \eqref{eq:thm_exp_cvgz} and \eqref{eq:Gstar-Gfuncstar} in Theorem \ref{thm:existence_minimizer_linear_convergence}.

\vspace{0.1cm}
\noindent\textit{4. $L^2$ exponential convergence of dual functions.} As for the third statement of Theorem \ref{thm:existence_minimizer_linear_convergence}, using Lemma \ref{lem:strong_convave_G} and \ref{lem:strong_convave_G2} , we deduce that
\begin{equation}
\begin{aligned}
    m_f \int_{\mathcal X}|f^n-f^*|^2d\mu &+m_g\int_{\mathcal Y}| g^{n}-g^*|^2d\nu +m_h\int_{\mathcal X}| h^{n}-h^*|^2d\mu\\
    &\le \mathcal G(f^{*},g^{*},h^{*})-  \mathcal G(f^{n},g^{n},h^{n}) \le C \big( 1-\frac{1}{C_s+1} \big)^n
\end{aligned}
\end{equation}
 for some constant $C>0$ and then the inequality \eqref{eq:L2convergence_formula} follows.

\vspace{0.1cm}
\noindent\textit{5. Optimality of induced martingale transport $\pi^*$.} Also due to the first-order conditions \eqref{eq:temp_h_FOC} and \eqref{eq:temp_fg_FOC}, it implies that $\pi(f^*,g^*,h^*)\in\mathcal M(\mu,\nu)$. 
As for the last argument, for any other martingale transport $ \pi'\in\mathcal M(\mu,\nu)$, it follows from the convexity of the mapping $\pi \mapsto H(\pi|Q)$ that
\begin{equation*}
\begin{aligned}
 H( \pi'|Q)  - &H\big(\pi(f^*,g^*,h^*)|Q\big)\ge \int_{\mathcal X\times \mathcal Y\times\mathcal Z} \big(\log(\pi(f^*,g^*,h^*))-\log Q\big)d( \pi' -  \pi(f^*,g^*,h^*)) \\
 & = \int_{\mathcal X\times \mathcal Y\times\mathcal Z} \big(-f^*(x)-g^*(y) - h^*(x) (y-x)\big)d( \pi' -  \pi(f^*,g^*,h^*)) = 0, \\
 &= \int_{\mathcal X} -f^*(x)(d\mu-d\mu) -\int_{\mathcal Y} g^*(y)(d\nu-d\nu) = 0, 
\end{aligned}
\end{equation*}
where the last equality holds because $\pi',\pi(f^*,g^*,h^*)\in\mathcal M(\mu,\nu)$. Thus, we have shown that the probability measure $\pi(f^*,g^*,h^*) \in \mathcal M(\mu,\nu)$ is indeed the solution to the primal EMOT problem \eqref{eq:mg_transport}. 
\end{proof}
\nocite{*}

% Appendix here
% Options are (1) APPENDIX (with or without general title) or
%             (2) APPENDICES (if it has more than one unrelated sections)
% Outcomment the appropriate case if necessary
%
% \begin{APPENDIX}{<Title of the Appendix>}
% \end{APPENDIX}
%
%   or
%
% \begin{APPENDICES}
% \section{<Title of Section A>}
% \section{<Title of Section B>}
% etc
% \end{APPENDICES}

% \begin{APPENDICES}
\section*{Appendix}
\subsection*{Well-definedness of Sinkhorn's steps}
\label{sec:well_defined_fgh}
\begin{proposition}
\label{prop:boundary_condition}
    { From the definition in Assumption \ref{assu:mu_nu},} under Assumption \ref{assu:exsitence_mg_transport_conti}, the essential bound of $\mu$, $\nu$ satisfies $\overline{X}<\overline{Y}$ and $\underline{X}>\underline{Y}$.
\end{proposition}

\begin{proof}{Proof}
We prove by contradiction, suppose that $\overline{X} \ge \overline{Y}$. From the martingale property of $\bar \pi$ we know 
    \begin{equation*}
    \begin{aligned}
                0 & = \int_{\mathbb R\times\mathbb R} (y-x) \exp(-p(x,y,z)) \nu(dy)\rho(dz) \\
                &\le \int_{\mathbb R\times\mathbb R} (y-x)^+ \exp(|p|_{L^\infty}) \nu(dy)\rho(dz) -\int_{\mathbb R\times\mathbb R} (y-x)^- \exp(-|p|_{L^\infty}) \nu(dy)\rho(dz), \\
    \end{aligned}
    \end{equation*}
for the right-hand side, we can push limit $x\to\overline{X}$ to see that it converges to a negative value since the measure $\nu(dy)$ is not a Dirac mass, which is a contradiction. The remaining part of this proposition can be shown similarly. 
\end{proof}

\begin{remark}
    Indeed, Assumption 2.3 in \cite{nutz2024martingale} describes the fact of Proposition \ref{prop:boundary_condition} as well, but under our setting it is implied by Assumption \ref{assu:exsitence_mg_transport_conti}.
\end{remark}

\begin{proof}{Proof of Proposition \ref{prop:well_defined_fgh}} 
    We are going to prove that $\tilde f^{n+1}$,$\tilde g^{n+1}$,$\tilde h^{n+1}$ are well defined, and $\tilde f^{n+1},\tilde h^{n+1}\in L^\infty(\mu)\cap  C^1(\mathcal X_0)$, $\tilde g^{n+1} \in L^\infty(\nu)\cap C^1(\mathcal Y_0)$ for $n\in\mathbb N$. The $L^\infty$ boundedness and continuity of $(f^{n+1},g^{n+1},h^{n+1})_{n}$ can be obtained by the definition  \eqref{eq:normalization_iter_n}. We prove this statement by induction. Suppose that it holds true for some
$n \ge 0$. Recalling the Sinkhorn's step in defining the function $\tilde{h}^{n+1}$ on $\mathcal X_0$, for any fixed $x\in\mathcal X_0$ it satisfies:
\begin{equation}
    \label{eq:h_update_in_proof}
    \begin{aligned}
    \tilde h^{n+1}(x) &:= h(x) 
    \end{aligned}
\end{equation}
such that
\begin{equation*}
    \Phi(x,h(x)) = \int_{\mathcal Y\times\mathcal Z} (y-x)\exp\big(-c(x,y,z)-{ g}^{n}(y)-h(x) (y-x)\big)(\nu\otimes\rho)(dy,dz) = 0.
\end{equation*}
    % \begin{equation}
    % \label{eq:h_update_in_proof}
    % \begin{aligned}
    % \tilde h^{n+1}(x) &:= h(x) \text{\ \ such\ that\ } \\
    % &\Phi(x,h(x)) = \int_{\mathcal Y\times\mathcal Z} (y-x)\exp\big(-c(x,y,z)-{ g}^{n}(y)-h(x) (y-x)\big)(\nu\otimes\rho)(dy,dz) = 0.
    % \end{aligned}
    % \end{equation}
    The function $\Phi(x,h)$ is continuous and non-increasing in $h$. 
    % tends to $\overline Y_i-x_i>0$ as $h_i(x)\to -\infty$, on the other direction, it tends to $\underline Y_i-x<0$ as $h_i(x)\to +\infty$. 
    Moreover, its limit behavior can be characterized as
    \begin{equation*}
        \lim_{h\to-\infty} \Phi(x,h) = +\infty \quad \text{and} \quad \lim_{h\to+\infty} \Phi(x,h) = -\infty.
    \end{equation*}
    Consequently, $\Phi(x,h)$ takes values across $0$. Hence, for each fixed $x\in\mathcal X_0$, the zero of the function $\Phi(x,h)$ exists and defines a finite value $\tilde h^{n+1}(x)$ pointwise. To establish the continuity of $h^{n+1}$, we apply the implicit function theorem (see, e.g., \cite{krantz2002implicit}). It follows that 
    \begin{equation*}
        \begin{aligned}
        \frac{\partial \Phi(x,\omega)}{\partial x} =  \int_{\mathcal Y\times\mathcal Z} \Big(-1 +(y-x)&\Big(-\frac{\partial c(x,y,z)}{\partial x}+\omega\Big)\Big) \\
        &\exp(-c(x,y,z)-{ g}^{n}(y)-\omega (y-x))(\nu\otimes\rho)(dy,dz),
        \end{aligned}
    \end{equation*}
    \begin{equation*}
        \frac{\partial \Phi(x,\omega)}{\partial \omega} =  -\int_{\mathcal Y\times\mathcal Z} |y-x|^2\exp(-c(x,y,z)-{ g}^{n}(y)-\omega(y-x))(\nu\otimes\rho)(dy,dz),
    \end{equation*}
    which are both continuous in $(x,\omega)$, and the derivative $\partial \Phi(x,\omega)/\partial \omega$ does not vanish anywhere.    
    The argument above shows that the zero point of $\omega\mapsto \Phi(x,\omega)$ exists for any $x\in\mathcal X_0$, so with the implicit function theorem, we proved $\tilde h^{n+1}\in \mathcal C^1(\mathcal X_0)$. To show $\tilde h^{n+1}\in L^{\infty}(\mu)$, we fix an $\epsilon>0$ to be sufficiently small such that  $\overline{X}+\epsilon < \overline{Y}$, which can be achieved  thanks to Proposition \ref{prop:boundary_condition}.  Note that for any $M>0$ there exists a constant $h_M<0$ such that for any $h\le h_M$, it holds
    \begin{equation}
    \label{eq:basic_ineq_exp_trun}
        \int_{[x+\epsilon,+\infty)} \exp(-h(y-x))\nu(dy) \ge M  \int_{(-\infty, x+\epsilon)} \exp(-h(y-x))\nu(dy).
    \end{equation}
    Hence, derived from the definition of function $\tilde h^{n+1}(x)$ we have 
    \begin{equation*}
    \begin{aligned}
    0 &=  \int_{\mathcal Y\times\mathcal Z} (y-x)\exp(-c(x,y,z)-g^n(y) - \tilde h^{n+1}(x)(y-x))(\nu\otimes\rho)(dy,dz)\\
    &\ge \int_{(-\infty,x]}(y-x)\exp(|c|_{L^\infty}+|g^n|_{L^\infty} - \tilde h^{n+1}(x)(y-x))\nu(dy)\\
    & + \int_{(x,x+\epsilon]}(y-x)\exp(-|c|_{L^\infty}-|g^n|_{L^\infty} -\tilde  h^{n+1}(x)(y-x))\nu(dy)\\
    & + \int_{(x+\epsilon,+\infty)} \epsilon\exp(-|c|_{L^\infty}-|g^n|_{L^\infty} - \tilde  h^{n+1}(x)(y-x))\nu(dy),
    \end{aligned}
    \end{equation*}
    and omit the second term because of the positivity. Hence, it should hold that
    \begin{equation*}
    \begin{aligned}
    (|\mathcal X| + |\mathcal Y|)\exp(2&|c|_{L^\infty}+2|g^n|_{L^\infty})\int_{(-\infty,x]}\exp(- \tilde h^{n+1}(x)(y-x))\nu(dy)\\
    &\ge  \int_{(x+\epsilon,+\infty)} \epsilon\exp( - \tilde  h^{n+1}(x)(y-x))\nu(dy).
    \end{aligned}
    \end{equation*}
Therefore if we denote $M:= (|\mathcal X| + |\mathcal Y|)\exp(2|c|_{L^\infty}+2|g^n|_{L^\infty})/\epsilon$, it can be shown that $\tilde h^{n+1}(x) > h_M$ otherwise it will violate \eqref{eq:basic_ineq_exp_trun}. Therefore, we find the lower bound for $\tilde h^{n+1}$ for any $x\in\mathcal X_0$, and the upper bound can be given similarly.

Finally, it remains to verify that $\tilde f^{n+1}\in L^\infty(\mu)\cap  C^1(\mathcal X_0)$ and $\tilde g^{n+1} \in L^\infty(\nu)\cap C^1(\mathcal Y_0)$.  From the explicit updating formulas for $\tilde{f}^{n+1}$ and $\tilde{g}^{n+1}$ (defined by \eqref{eq:update_formula_fg}), it directly follows that both functions are bounded and continuously differentiable, given the established fact that $\tilde h^{n+1}\in L^\infty(\mu)\cap C^1(\mathcal X_0)$. 
\end{proof}

% \end{APPENDICES}

\bibliographystyle{plain}
\bibliography{ref}

@article{nutz2024martingale,
  title={On the Martingale Schr\"odinger Bridge between Two Distributions},
  author={Nutz, Marcel and Wiesel, Johannes},
  journal={arXiv preprint arXiv:2401.05209},
  year={2024}
}

@article{henry2019martingale,
  title={From (Martingale) Schrodinger bridges to a new class of Stochastic Volatility Models},
  author={Henry-Labordere, Pierre},
  journal={arXiv preprint arXiv:1904.04554},
  year={2019}
}

@article{guyon2024dispersion,
  title={Dispersion-constrained martingale Schr{\"o}dinger problems and the exact joint S\&P 500/VIX smile calibration puzzle},
  author={Guyon, Julien},
  journal={Finance and Stochastics},
  volume={28},
  number={1},
  pages={27--79},
  year={2024},
  publisher={Springer}
}

@book{krantz2002implicit,
  title={The implicit function theorem: history, theory, and applications},
  author={Krantz, Steven George and Parks, Harold R},
  year={2002},
  publisher={Springer Science \& Business Media}
}

@article{kulpa1997poincare,
  title={The poincar{\'e}-miranda theorem},
  author={Kulpa, Wladyslaw},
  journal={The American Mathematical Monthly},
  volume={104},
  number={6},
  pages={545--550},
  year={1997},
  publisher={Taylor \& Francis}
}

@article{de2018entropic,
  title={Entropic approximation for multi-dimensional martingale optimal transport},
  author={De March, Hadrien},
  journal={arXiv preprint arXiv:1812.11104},
  year={2018}
}

@book{schrodinger1931umkehrung,
  title={{\"U}ber die umkehrung der naturgesetze},
  author={Schr{\"o}dinger, Erwin},
  year={1931},
  publisher={Springer Science and Business Media LLC}
}

@article{beiglbock2013model,
  title={Model-independent bounds for option prices—a mass transport approach},
  author={Beiglb{\"o}ck, Mathias and Henry-Labordere, Pierre and Penkner, Friedrich},
  journal={Finance and Stochastics},
  volume={17},
  pages={477--501},
  year={2013},
  publisher={Springer}
}

@article{galichon2014stochastic,
  title={A stochastic control approach to no-arbitrage bounds given marginals, with an application to lookback options},
  author={Galichon, Alfred and Henry-Labordere, Pierre and Touzi, Nizar},
  journal={The Annals of Applied Probability},
  year={2014}
}

@article{song2012tale,
  title={A tale of two option markets: State-price densities implied from S\&P 500 and VIX option prices},
  author={Song, Zhaogang and Xiu, Dacheng},
  journal={Unpublished working paper. Federal Reserve Board and University of Chicago},
  year={2012},
  publisher={Citeseer}
}

@article{jacquier2018vix,
  title={On VIX futures in the rough Bergomi model},
  author={Jacquier, Antoine and Martini, Claude and Muguruza, Aitor},
  journal={Quantitative Finance},
  volume={18},
  number={1},
  pages={45--61},
  year={2018},
  publisher={Taylor \& Francis}
}

@book{gatheral2011volatility,
  title={The volatility surface: a practitioner's guide},
  author={Gatheral, Jim},
  year={2011},
  publisher={John Wiley \& Sons}
}

@book{villani2009optimal,
  title={Optimal transport: old and new},
  author={Villani, C{\'e}dric and others},
  volume={338},
  year={2009},
  publisher={Springer}
}

@article{santambrogio2015optimal,
  title={Optimal transport for applied mathematicians},
  author={Santambrogio, Filippo},
  journal={Birk{\"a}user, NY},
  volume={55},
  number={58-63},
  pages={94},
  year={2015},
  publisher={Springer}
}

@inproceedings{arjovsky2017wasserstein,
  title={Wasserstein generative adversarial networks},
  author={Arjovsky, Martin and Chintala, Soumith and Bottou, L{\'e}on},
  booktitle={International conference on machine learning},
  pages={214--223},
  year={2017},
  organization={PMLR}
}

@article{flamary2021pot,
  title={Pot: Python optimal transport},
  author={Flamary, R{\'e}mi and Courty, Nicolas and Gramfort, Alexandre and Alaya, Mokhtar Z and Boisbunon, Aur{\'e}lie and Chambon, Stanislas and Chapel, Laetitia and Corenflos, Adrien and Fatras, Kilian and Fournier, Nemo and others},
  journal={Journal of Machine Learning Research},
  volume={22},
  number={78},
  pages={1--8},
  year={2021}
}

@article{strassen1965existence,
  title={The existence of probability measures with given marginals},
  author={Strassen, Volker},
  journal={The Annals of Mathematical Statistics},
  volume={36},
  number={2},
  pages={423--439},
  year={1965},
  publisher={Institute of Mathematical Statistics}
}

@article{dolinsky2014martingale,
  title={Martingale optimal transport and robust hedging in continuous time},
  author={Dolinsky, Yan and Soner, H Mete},
  journal={Probability Theory and Related Fields},
  volume={160},
  number={1},
  pages={391--427},
  year={2014},
  publisher={Springer}
}

@article{cheridito2021martingale,
  title={Martingale optimal transport duality},
  author={Cheridito, Patrick and Kiiski, Matti and Pr{\"o}mel, David J and Soner, H Mete},
  journal={Mathematische Annalen},
  volume={379},
  pages={1685--1712},
  year={2021},
  publisher={Springer}
}

@article{guo2016monotonicity,
  title={On the monotonicity principle of optimal Skorokhod embedding problem},
  author={Guo, Gaoyue and Tan, Xiaolu and Touzi, Nizar},
  journal={SIAM Journal on Control and Optimization},
  volume={54},
  number={5},
  pages={2478--2489},
  year={2016},
  publisher={SIAM}
}

@article{hou2018robust,
  title={Robust pricing--hedging dualities in continuous time},
  author={Hou, Zhaoxu and Ob{\l}{\'o}j, Jan},
  journal={Finance and Stochastics},
  volume={22},
  number={3},
  pages={511--567},
  year={2018},
  publisher={Springer}
}

@article{cheridito2017duality,
  title={Duality formulas for robust pricing and hedging in discrete time},
  author={Cheridito, Patrick and Kupper, Michael and Tangpi, Ludovic},
  journal={SIAM Journal on Financial Mathematics},
  volume={8},
  number={1},
  pages={738--765},
  year={2017},
  publisher={SIAM}
}

@article{bartl2019duality,
  title={Duality for pathwise superhedging in continuous time},
  author={Bartl, Daniel and Kupper, Michael and Pr{\"o}mel, David J and Tangpi, Ludovic},
  journal={Finance and Stochastics},
  volume={23},
  number={3},
  pages={697--728},
  year={2019},
  publisher={Springer}
}

@article{beiglbock2017optimal,
  title={Optimal transport and Skorokhod embedding},
  author={Beiglb{\"o}ck, Mathias and Cox, Alexander MG and Huesmann, Martin},
  journal={Inventiones mathematicae},
  volume={208},
  pages={327--400},
  year={2017},
  publisher={Springer}
}

@article{beiglbock2021fine,
  title={Fine properties of the optimal Skorokhod embedding problem},
  author={Beiglb{\"o}ck, Mathias and Nutz, Marcel and Stebegg, Florian},
  journal={Journal of the European Mathematical Society},
  volume={24},
  number={4},
  pages={1389--1429},
  year={2021}
}

@article{wiesel2019continuity,
  title={Continuity of the martingale optimal transport problem on the real line},
  author={Wiesel, Johannes},
  journal={arXiv preprint arXiv:1905.04574},
  year={2019}
}

@article{backhoff2022stability,
  title={Stability of martingale optimal transport and weak optimal transport},
  author={Backhoff-Veraguas, Julio and Pammer, Gudmund},
  journal={The Annals of Applied Probability},
  volume={32},
  number={1},
  pages={721--752},
  year={2022},
  publisher={Institute of Mathematical Statistics}
}

@article{sinkhorn1967concerning,
  title={Concerning nonnegative matrices and doubly stochastic matrices},
  author={Sinkhorn, Richard and Knopp, Paul},
  journal={Pacific Journal of Mathematics},
  volume={21},
  number={2},
  pages={343--348},
  year={1967},
  publisher={Mathematical Sciences Publishers}
}

@article{peyre2019computational,
  title={Computational optimal transport: With applications to data science},
  author={Peyr{\'e}, Gabriel and Cuturi, Marco and others},
  journal={Foundations and Trends{\textregistered} in Machine Learning},
  volume={11},
  number={5-6},
  pages={355--607},
  year={2019},
  publisher={Now Publishers, Inc.}
}

@article{ghosal2022convergence,
  title={On the Convergence Rate of Sinkhorn's Algorithm},
  author={Ghosal, Promit and Nutz, Marcel},
  journal={arXiv preprint arXiv:2212.06000},
  year={2022}
}

@article{cuturi2013sinkhorn,
  title={Sinkhorn distances: Lightspeed computation of optimal transport},
  author={Cuturi, Marco},
  journal={Advances in neural information processing systems},
  volume={26},
  year={2013}
}

@article{peyre2019quantum,
  title={Quantum entropic regularization of matrix-valued optimal transport},
  author={Peyr{\'e}, Gabriel and Chizat, Lenaic and Vialard, Fran{\c{c}}ois-Xavier and Solomon, Justin},
  journal={European Journal of Applied Mathematics},
  volume={30},
  number={6},
  pages={1079--1102},
  year={2019},
  publisher={Cambridge University Press}
}

@article{solomon2015convolutional,
  title={Convolutional wasserstein distances: Efficient optimal transportation on geometric domains},
  author={Solomon, Justin and De Goes, Fernando and Peyr{\'e}, Gabriel and Cuturi, Marco and Butscher, Adrian and Nguyen, Andy and Du, Tao and Guibas, Leonidas},
  journal={ACM Transactions on Graphics (ToG)},
  volume={34},
  number={4},
  pages={1--11},
  year={2015},
  publisher={ACM New York, NY, USA}
}

@article{guo2019computational,
  title={Computational methods for martingale optimal transport problems},
  author={Guo, Gaoyue and Ob{\l}{\'o}j, Jan},
  journal={The Annals of Applied Probability},
  volume={29},
  number={6},
  pages={3311--3347},
  year={2019},
  publisher={JSTOR}
}

@article{carlier2022linear,
  title={On the linear convergence of the multimarginal Sinkhorn algorithm},
  author={Carlier, Guillaume},
  journal={SIAM Journal on Optimization},
  volume={32},
  number={2},
  pages={786--794},
  year={2022},
  publisher={SIAM}
}

@article{chen2016entropic,
  title={Entropic and displacement interpolation: a computational approach using the Hilbert metric},
  author={Chen, Yongxin and Georgiou, Tryphon and Pavon, Michele},
  journal={SIAM Journal on Applied Mathematics},
  volume={76},
  number={6},
  pages={2375--2396},
  year={2016},
  publisher={SIAM}
}

@article{franklin1989scaling,
  title={On the scaling of multidimensional matrices},
  author={Franklin, Joel and Lorenz, Jens},
  journal={Linear Algebra and its applications},
  volume={114},
  pages={717--735},
  year={1989},
  publisher={Elsevier}
}

@article{nutz2021introduction,
  title={Introduction to entropic optimal transport},
  author={Nutz, Marcel},
  journal={Lecture notes, Columbia University},
  year={2021}
}

@article{nutz2023stability,
  title={Stability of Schr{\"o}dinger potentials and convergence of Sinkhorn’s algorithm},
  author={Nutz, Marcel and Wiesel, Johannes},
  journal={The Annals of Probability},
  volume={51},
  number={2},
  pages={699--722},
  year={2023},
  publisher={Institute of Mathematical Statistics}
}

@article{ruschendorf1995convergence,
  title={Convergence of the iterative proportional fitting procedure},
  author={Ruschendorf, Ludger},
  journal={The Annals of Statistics},
  pages={1160--1174},
  year={1995},
  publisher={JSTOR}
}

@article{guyon2022fast,
  title={Fast exact joint S\&P 500/VIX smile calibration in discrete and continuous time},
  author={Guyon, Julien and Bourgey, Florian},
  journal={Available at SSRN 4315084},
  year={2022}
}

@article{de2019building,
  title={Building arbitrage-free implied volatility: Sinkhorn's algorithm and variants},
  author={De March, Hadrien and Henry-Labordere, Pierre},
  journal={arXiv preprint arXiv:1902.04456},
  year={2019}
}

@article{guyon2020joint,
  title={The joint S\&P 500/VIX smile calibration puzzle solved},
  author={Guyon, Julien},
  journal={Risk, April},
  year={2020}
}

@article{doldi2024entropy,
  title={On entropy martingale optimal transport theory},
  author={Doldi, Alessandro and Frittelli, Marco and Rosazza Gianin, Emanuela},
  journal={Decisions in Economics and Finance},
  pages={1--42},
  year={2024},
  publisher={Springer}
}

@article{doldi2023entropy,
  title={Entropy martingale optimal transport and nonlinear pricing--hedging duality},
  author={Doldi, Alessandro and Frittelli, Marco},
  journal={Finance and Stochastics},
  volume={27},
  number={2},
  pages={255--304},
  year={2023},
  publisher={Springer}
}

@article{hull1987pricing,
  title={The pricing of options on assets with stochastic volatilities},
  author={Hull, John and White, Alan},
  journal={The journal of finance},
  volume={42},
  number={2},
  pages={281--300},
  year={1987},
  publisher={Wiley Online Library}
}

@article{chesney1989pricing,
  title={Pricing European currency options: A comparison of the modified Black-Scholes model and a random variance model},
  author={Chesney, Marc and Scott, Louis},
  journal={Journal of Financial and Quantitative Analysis},
  volume={24},
  number={3},
  pages={267--284},
  year={1989},
  publisher={Cambridge University Press}
}

@article{heston1993closed,
  title={A closed-form solution for options with stochastic volatility with applications to bond and currency options},
  author={Heston, Steven L},
  journal={The review of financial studies},
  volume={6},
  number={2},
  pages={327--343},
  year={1993},
  publisher={Oxford University Press}
}

@article{engle1982autoregressive,
  title={Autoregressive conditional heteroscedasticity with estimates of the variance of United Kingdom inflation},
  author={Engle, Robert F},
  journal={Econometrica: Journal of the econometric society},
  pages={987--1007},
  year={1982},
  publisher={JSTOR}
}

@book{karlin1981second,
  title={A second course in stochastic processes},
  author={Karlin, Samuel and Taylor, Howard E},
  year={1981},
  publisher={Elsevier}
}

@article{hiew2024ordinary,
  title={An ordinary differential equation for entropic optimal transport and its linearly constrained variants},
  author={Hiew, Joshua Zoen-Git and Nenna, Luca and Pass, Brendan},
  journal={arXiv preprint arXiv:2403.20238},
  year={2024}
}

@article{galichon2022cupid,
  title={Cupid’s invisible hand: Social surplus and identification in matching models},
  author={Galichon, Alfred and Salani{\'e}, Bernard},
  journal={The Review of Economic Studies},
  volume={89},
  number={5},
  pages={2600--2629},
  year={2022},
  publisher={Oxford University Press}
}

@article{benamou2015iterative,
  title={Iterative Bregman projections for regularized transportation problems},
  author={Benamou, Jean-David and Carlier, Guillaume and Cuturi, Marco and Nenna, Luca and Peyr{\'e}, Gabriel},
  journal={SIAM Journal on Scientific Computing},
  volume={37},
  number={2},
  pages={A1111--A1138},
  year={2015},
  publisher={SIAM}
}

@article{chizat2018scaling,
  title={Scaling algorithms for unbalanced optimal transport problems},
  author={Chizat, Lenaic and Peyr{\'e}, Gabriel and Schmitzer, Bernhard and Vialard, Fran{\c{c}}ois-Xavier},
  journal={Mathematics of computation},
  volume={87},
  number={314},
  pages={2563--2609},
  year={2018}
}

@article{schmitzer2019stabilized,
  title={Stabilized sparse scaling algorithms for entropy regularized transport problems},
  author={Schmitzer, Bernhard},
  journal={SIAM Journal on Scientific Computing},
  volume={41},
  number={3},
  pages={A1443--A1481},
  year={2019},
  publisher={SIAM}
}

@article{beck2013convergence,
  title={On the convergence of block coordinate descent type methods},
  author={Beck, Amir and Tetruashvili, Luba},
  journal={SIAM journal on Optimization},
  volume={23},
  number={4},
  pages={2037--2060},
  year={2013},
  publisher={SIAM}
}
% \printbibliography

\end{document}